\documentclass[article]{amsart}
\usepackage{cases}
\usepackage{caption}
\usepackage{amsmath, amsfonts, pifont, amssymb}
\usepackage{verbatim}
\usepackage[bookmarks=true]{hyperref}
\usepackage{tikz}
\usetikzlibrary{positioning}
\usetikzlibrary{calc}
\usetikzlibrary{snakes,arrows}
\usetikzlibrary{quotes,angles}
\usepackage{geometry}
\usepackage{color}
\usepackage[nameinlink,noabbrev]{cleveref}
\newtheorem{theorem}{Theorem}[section]

\newtheorem{lemma}[theorem]{Lemma}
\newtheorem{proposition}[theorem]{Proposition}
\newtheorem{corollary}[theorem]{Corollary}

\newcommand{\R}{\mathbb{R}}

\newcommand{\C}{\mathbb{C}}

\newcommand{\beq}{\begin{equation}}
	\newcommand{\eeq}{\end{equation}}
\newcommand{\beqq}{\begin{equation*}}
	\newcommand{\eeqq}{\end{equation*}}

\newcommand{\bfu}{\mathbf{u}}
\theoremstyle{definition}
\newtheorem{definition}[theorem]{Definition}

\theoremstyle{remark}
\newtheorem{remark}[theorem]{Remark}

\numberwithin{equation}{section}
\newcommand{\T}{\mathbb{T}}
\newcommand{\Z}{\mathbb{Z}}
\newcommand{\dd}{\,\mathrm d}
\newcommand{\abs}[1]{\lvert#1\rvert}
\newcommand{\norm}[1]{\|#1\|}

\DeclareDocumentCommand{\abs}{s m}{
	\operatorname{}
	\IfBooleanTF{#1}{#2}{\left|#2\right|}}

\DeclareDocumentCommand{\norm}{s m}{
	\operatorname{}
	\IfBooleanTF{#1}{#2} {\left\| #2\right\|}}

\DeclareDocumentCommand{\inner}{s m}{
	\operatorname{}
	\IfBooleanTF{#1}{#2}{\left \langle#2\right \rangle}}

\DeclareDocumentCommand{\parenthese}{s m}{
	\operatorname{}
	\IfBooleanTF{#1}{#2}{\left(#2\right)}}

\DeclareDocumentCommand{\square}{s m}{
	\operatorname{}
	\IfBooleanTF{#1} {#2}{\left[#2\right]}}

\DeclareDocumentCommand{\bracket}{s m}{
	\operatorname{}
	\IfBooleanTF{#1}{#2}{\left\{#2\right\}}}

\begin{document}
	\author{Yilin Song}
	\address{Yilin Song
	\newline \indent  Graduate School of China Academy of Engineering Physics,
	 \newline \indent Beijing 100088,\ People’s Republic of China.}
	 \email{songyilin21@gscaep.ac.cn}
     	\author{Yan Xue}
	\address{Yan Xue
	\newline \indent Laboratory of Mathematics and Complex Systems,
\newline \indent Ministry of Education,
\newline \indent School of Mathematical Sciences,
\newline \indent Beijing Normal University,
\newline \indent Beijing 100875,\ People’s Republic of China.}
	 \email{yanxue0789@gmail.com}
    	\author{Ruixiao Zhang}
	\address{Ruixiao Zhang
	\newline \indent Department of Mathematics,
\newline \indent Northeastern University,
\newline \indent Shenyang 110819,\ People’s Republic of China.}
	 \email{zhangrx@mail.neu.edu.cn} 
	
\title[quadratic Schr\"odinger system]{ The quadratic Schr\"odinger system on $\mathbb{T}^6$
 }
\begin{abstract}
	In this article,  we study the global well-posedness for 
      the quadratic Schr\"odinger system on $\mathbb{T}^6$. 
      To achieve this, we utilize the concentration-compactness/rigitdity argument initiated by Ionescu and Pausader \cite{IPT3}. 
      Our result is also regarded as a periodic analogue of Gao-Meng-Xu-Zheng \cite{GMXZ}.
\end{abstract}
\maketitle
\begin{center}
	\begin{minipage}{100mm}
	\end{minipage}
\end{center}
\tableofcontents
\section{Introduction}
In this article, let $\mathbb T^6=\mathbb R^6/(2\pi\mathbb Z)^6$, we consider
\begin{equation}\label{eq:system}
\begin{cases}
(i\partial_t+\Delta)u_1=-u_2\overline{u_1},\\
(i\partial_t+\kappa \Delta)u_2=-u_1^2,\\
(u_1(0),u_2(0))=(u_{1,0},u_{2,0})\in H^1(\mathbb T^6;\mathbb C^2).
\end{cases}
\end{equation}
When $\kappa=\frac{1}{2}$, we notice that
\eqref{eq:system} is left invariant under the
Galilean transformation
\[
\begin{pmatrix}
u_1(t,x)\\
u_2(t,x)
\end{pmatrix}
\longmapsto
\begin{pmatrix}
e^{ix\cdot\xi}e^{-it|\xi|^2}u_1(t,x-2t\xi)\\
e^{2ix\cdot\xi}e^{-2it|\xi|^2}u_2(t,x-2t\xi)
\end{pmatrix}
\]
for any $\xi\in\mathbb{R}^d$. In this situation, we say  that \eqref{eq:system} satisfies the
mass-resonance condition. In this paper, we set $\kappa =\frac{1}{2}$, i.e. mass-resonance. For this setting, one can refer to \cite{HOT2013} to get more details.

In vector expressions, we define the solution ${\bf u}:I\times\T^6\to\mathbb{C}^2$ by $\mathbf{u}=(u_1,u_2)$ and $F(\mathbf{u})=(u_1\overline{u_2},u_1^2)$.
Thus, we rewrite \eqref{eq:system} as
\begin{equation}\label{eq:system1}
\begin{cases}
(i\partial_t+\Delta)\mathbf{u}=-F(\mathbf{u}),\\
\mathbf{u}(0)=\mathbf{u}_0\in H^1(\mathbb T^6;\mathbb C^2).
\end{cases}
\end{equation}
\subsection{Background}
We consider the nonlinear Schrödinger equation
\begin{equation}\label{eq:nls}
\begin{cases}
(i\partial_t+\Delta)u=\mu |u|^{p-1}u,\qquad t \in \mathbb{R}^d,x\in \mathcal{M}^d,\\
u(0,x)=u_0(x) \in \dot{H}^s(\mathcal{M}^d)
\end{cases}
\end{equation}
where $\mathcal{M}^d$ is a compact manifold in dimension $d$. We mainly focusing on $\mathcal{M}= \mathbb{R}$ or $\mathbb{T}$ in this paper.
We call the \eqref{eq:nls} is focusing if $\mu =-1$ and defocusing otherwise. Under scaling $u(t,x) \mapsto u_{\lambda}:=\lambda^{\frac{2}{p-1}}u(\lambda^2t,\lambda x)$, we know that the critical exponent of \eqref{eq:nls} in $\mathbb{R}^d$ is 
\[s_c=\frac{d}{2}-\frac{2}{p-1},\]
which satisfies $\|u\|_{\dot{H}^{s_c}}=\|u_{\lambda}\|_{\dot{H}^{s_c}}$.
We continue to use this exponent in $\mathbb{T}^d$. We call \eqref{eq:nls} is a critical problem if $s=s_c$, and sub-critical if $s <s_c$.
By this convention, it is easy to check that \eqref{eq:system1} is $\dot{H}^1(\mathbb{T}^6)$-critical. In this paper, we focus on global wellposedness of $\eqref{eq:system1}$. Before starting with this, 
let's recall some history of the nonlinear Schrödinger equation and the quadratic Schr\"odinger system.
\subsubsection{Well-posedness results in $\mathbb R^d$}
For the local well-posedness theory of nonlinear Schrödinger equations on \(\mathbb R^d\), one can refer to Cazenave \cite{Cazenave2003}. For the defocusing energy-critical NLS on $\mathbb R^d$,
global well-posedness and scattering for arbitrary initial data
in $\dot H^1(\mathbb R^d)$ are known in dimensions $d\ge3$.
Early results for radial data were obtained by Bourgain
\cite{Bourgain1999}.
Without the radial assumption, the problem was resolved by
Colliander, Keel, Staffilani, Takaoka, and Tao \cite{CKSTT2008}
in dimension $d=3$, Ryckman and Vi\c{s}an \cite{RyckmanVisan2007}
in dimension $d=4$, and Vi\c{s}an \cite{Visan2007}
in dimensions $d\ge 5$.
For the focusing energy-critical equation in the radial setting, Kenig and Merle \cite{KenigMerle2006} introduced the concentration-compactness and rigidity approach to prove scattering below the ground state. 
For nonradial data, global well-posedness and scattering below the
ground-state threshold were established by Killip and Vi\c{s}an
\cite{KillipVisan2010} in dimensions $d\ge5$ and by Dodson
\cite{Dodson2019} in dimension $d=4$. Very recently, Su and Zhao \cite{SuZhao2026} solved this problem in dimension $d=3$.

For the quadratic system, the scaling $\mathbf u(t,x)\mapsto\lambda^2\mathbf u(\lambda^2t,\lambda x)$ gives the critical index $s_c=d/2-2$. 
Thus, $d=4$ and $d=6$ correspond to the mass-critical and energy-critical problems, respectively. 
Hayashi, Ozawa and Tanaka \cite{HOT2013} established local well-posedness for the quadratic system in $L^2(\mathbb R^d;\mathbb C^2)$ for $1\le d\le4$ and in $H^1(\mathbb R^d;\mathbb C^2)$ for $1\le d\le6$. 

At critical regularity, Hayashi, Li, and Ozawa \cite{HLO2011} proved global existence and scattering for data in $H^{s_c}$ with sufficiently small $\dot H^{s_c}$ norm when $d\ge4$. 
In dimension four, Hayashi, Ozawa, and Tanaka \cite{HOT2013}
used a sharp Gagliardo--Nirenberg inequality to prove global
well-posedness in $H^1$ for initial data whose mass is below
the ground state.
Inui, Kishimoto, and Nishimura \cite{IKN2019} established
global well-posedness and scattering in $L^2$ under the same
mass condition.
In dimensions $d=5$, Hamano \cite{Hamano2018} proved scattering
in $H^1$ for the mass-resonant system without assuming radial
symmetry below the ground-state.
Under the same threshold conditions, Hamano, Inui and
Nishimura \cite{HIN2021} proved scattering for radial data
without mass-resonance condition.
Wang and Yang \cite{WangYang2019} obtained scattering for
nonradial data when $\kappa$ is sufficiently close to $\frac12$,
and for radial data for every $\kappa>0$, using the approach
of Dodson and Murphy \cite{DodsonMurphy2017,DodsonMurphy2018}.
Meng and Xu \cite{MengXu2021} gave an alternative proof under the mass-resonance  using the Dodson-Murphy strategy.
The energy-critical quadratic system  was studied by Gao, Meng, Xu and Zheng \cite{GMXZ}.
They proved global well-posedness and scattering in $\dot H^1(\mathbb R^6;\mathbb C^2)$ for data whose energy and kinetic energy are both below the ground state. 
The result holds for general data under mass resonance and for radial data without this restriction. 
\subsubsection{Well-posedness results in $\mathbb T^d$}
In the periodic setting, the lack of dispersive effect makes the local well-posedenss extremely difficult compared to the Eulicdean setting. So the main difficult is to establish Strichartz estimates on tours. A remarkable breakthrough is that Bourgain \cite{Bour1993} established the periodic Strichartz estimates using the by the analytical number theory, and obtain some relevant local well-posedness results. However, the Strichartz estimates in their paper are almost sharp  for lower dimensions $d=1,2$ up to the $\varepsilon$-loss and far from optimal in $d\geq3$. 
Then Bourgain and Demeter \cite{BD} proved the $\ell^2$ decoupling theorem, 
yielding the expected range of periodic Strichartz estimates with an arbitrarily small frequency loss in all dimensions $d\geq1$ on both rational and irrational tours. 
Recently, Killip and Vi\c{s}an \cite{KillipVisan2016} obtained scale-invariant estimates on rectangular torus above the critical Strichartz exponent. Relevant results can refer to \cite{DGG2017,FSWW2018,GOW2014,Strunk2014}.

Considering global well-posednesss of critical problems, there are a series of excellent results recently. For energy-critical NLS on the rational three-dimensional torus, Herr, Tataru, and Tzvetkov \cite{HerrT32011} proved global well-posedness for small $H^1$ data. 
Their proof combines a refined trilinear estimate with critical spaces constructed from $U^p$ and $V^p$, 
which was developed by Hadac, Herr, and Koch \cite{HadacKp22009}.These spaces are suitable to study critical problems. The corresponding estimates on irrational torus were further developed in \cite{GOW2014,Strunk2014}. For large data, Ionescu and Pausader \cite{IPT3} established global well-posedness for the defocusing energy-critical NLS on the rational torus $\mathbb T^3$, 
and Strunk \cite{Strunk2015} extended the result to general rectangular three-dimensional torus. The key is to use Euclidean scattering results to control concentrating solutions on the torus.
Related arguments were developed on $\mathbb R\times\mathbb T^3$ by Ionescu and Pausader \cite{IP2012Waveguide} and on hyperbolic space by Ionescu, Pausader and Staffilani \cite{IPS2012}. 
Yu and Yue \cite{YY} studied the defocusing quintic NLS on $\mathbb T^2$ and proved global existence provided the critical $H^{\frac{1}{2}}$ norm remains bounded throughout its maximal lifespan. For the focusing case, Yue \cite{Yue2021} proved global existence for the cubic NLS on $\mathbb T^4$ below the Euclidean threshold. 

In this paper, we study global well-posedness of the quadratic Schr\"odinger system.
\subsection{Main results}
First, let us recall some basic facts related to our problem. For the smooth solution to \eqref{eq:system1}, the mass, energy and kintic energy quantities can be defined as following,
\begin{align}
M(\mathbf{u})&=\|u_1\|_2^2+\|u_2\|_2^2,\label{mass}\\
E(\mathbf{u})&=\|\nabla u_1\|_2^2+\frac14\|\nabla u_2\|_2^2
       -\operatorname{Re}\int\overline{u_2}u_1^2,\label{energy}\\
H(\mathbf{u})&= \|\nabla u_1\|_2^2+\frac14\|\nabla u_2\|_2^2
\end{align}
By a simple computation, $M(\mathbf{u})$ and $E(\mathbf{u})$ is conserved if $\mathbf{u}_0$ is smooth enough.

Also, we have ground state for $\eqref{eq:system1}$:
\begin{equation}\label{ground} Q(x)=(1+|x|^2/24)^{-2},\qquad \mathbf{W}=(Q/\sqrt2,Q),\qquad x\in\mathbb R^6,\end{equation}
which means that $\mathbf{W}$ satisfies  $\Delta \mathbf{W} =-F(\mathbf{W}).$

In Euclidean space, we have classical Sobolev embedding. Let $C_6$ denote the optimal constant in the Euclidean
Sobolev inequality
$\|f\|_{L^3(\mathbb R^6)}
\le C_6\|\nabla f\|_{L^2(\mathbb R^6)}$, it satisfies $
C_6^{-6}=4E(\mathbf W)=\frac43H(\mathbf W).$
However, for $\mathbb{T}^6$, we have to modified it by adding a mass term. By the sharp Sobolev inequality on a compact manifold (see  Hebey--Vaugon
\cite{HV}), there exists
$c_0=c_0(\mathbb T^6)>0$ such that
\begin{equation}\label{eq:torus-sharp-sobolev}
\|f\|_{L^3(\mathbb T^6)}^2
\le C_6^2\left(
\|\nabla f\|_{L^2(\mathbb T^6)}^2
+c_0\|f\|_{L^2(\mathbb T^6)}^2
\right),
\qquad f\in H^1(\mathbb T^6).
\end{equation}
Now, we can state the main result.
\begin{theorem}\label{thm:main}
There exists $c_0=c_0(\mathbb T^6)>0$ to make \eqref{eq:torus-sharp-sobolev} hold. If
\begin{equation}\label{threshold}
\begin{split}
E(\mathbf{u}_0)+c_0 M(\mathbf{u}_0)&<E(\mathbf{W}),\\
H(\mathbf{u_0})+c_0 M(\mathbf{u}_0)&<H(\mathbf{W}),
\end{split}
\end{equation}
then \eqref{eq:system1} has a unique global strong solution
\[
\mathbf{u}\in C(\mathbb R;H^1(\mathbb T^6))\cap X^1_{\mathrm{loc}}(\mathbb R).
\]
\end{theorem}
\subsection{Organization of the paper}
The remainder of this paper is organized as follows.
In \Cref{se2}, we introduce the function spaces and collect the
basic estimates used throughout the paper.
\Cref{se3} proves local
well-posedness and stability theory of \eqref{eq:system1} by establishing the multilinear estimates. In \Cref{se4}, we study Euclidean profiles and show that concentrating periodic solutions can be
approximated by rescaled Euclidean solutions.
\Cref{section5} proves the linear profile decomposition in
$H^1(\mathbb T^6)$ and the almost orthogonality of
nonlinear profiles.
Finally, in \Cref{se6}, we prove 
\Cref{thm:main} by a standard rigidity argument.

\section{Preliminaries}\label{se2}
In this section, we collect the basic definitions and properties of functional spaces, Strichartz estimates. Then, we use them to show the local well-posedness and stability of the quadratic system \eqref{eq:system1}.
\subsection{Definition of the strong solution}
For $\mathbf{u}=(u_1,u_2):I\times\mathcal M\to\mathbb C^2$, we write
\[
 \|\mathbf{u}\|_{L_t^qL_x^r(I\times\mathcal M)}
 =\left[\int_I\left(\int_{\mathcal M}|\bfu(t,x)|^r\,dx\right)^{q/r}dt\right]^{1/q},
 \qquad |\bfu|^2:=|u_1|^2+|u_2|^2.
\]
We set $\langle\xi\rangle=(1+|\xi|^2)^{1/2}$ and
\[
 \|\bfu\|_{H^s}^2=\sum_{j=1}^2\sum_{\xi\in\mathbb Z^6}
 \langle\xi\rangle^{2s}|\widehat u_j(\xi)|^2.
\]
The operators $P_N$ are Littlewood--Paley projections, which means that frequency of function is restrited to $\sim N$.

\begin{definition}[Solution]\label{solution}
Given an interval $I\subset\mathbb R$, a strong solution of $(i\partial_t+\Delta)\mathbf{u}=-F(\mathbf{u})$ is a function $\bfu\in C(I:H^1(\mathbb T^6))\cap X^1(I)$ such that, for every $t,s\in I$,
\[
 \bfu(t)=e^{i(t-s)\Delta}\bfu(s)+i\int_s^t e^{i(t-t')\Delta}F(\bfu(t'))\,dt'.
\]
\end{definition}

\subsection{Function spaces and linear Strichartz estimates}
In this subsection, we introduce the $X^s$ and $Y^s$ spaces based on the atomic spaces $U^p$ and $V^p$, which were originally developed by Tataru \cite{HadacKp22009,HerrT32011,HerrStri2014}. For convenience, we denote $\mathcal{H}$ the separable Hilbert space on $\C$ and also denote $\mathcal{Z}$ the set of finite partitions
$-\infty = t_0<t_1<\ldots<t_{K} = \infty$ on the real line. Furthermore, we set $v(\infty) := 0$ for any function $v:\R \rightarrow \mathcal{H}$.

We recall the definitions in \cite{HerrT32011} in the following:

\begin{definition}[$U^p$ space \cite{HerrT32011}]
	Let $1\leq p < \infty$ and functions $u: \R \rightarrow \mathcal{H}$. For sequence $\{t_k\}_{k=0}^K \in \mathcal{Z}$ and a sequence of functions $\{\phi_k\}_{k=0}^{K-1} \subset \mathcal{H}$ with $\sum_{k=0}^{K-1} \|\phi_k\|_{\mathcal{H}}^p = 1$ and $\phi_0 = 0$. A $U^p$-atom If a piecewise defined function $a : \mathbb{R} \to \mathcal{H}$ has the form
	\begin{equation*}
		a = \sum_{k=1}^K 1_{[t_{k-1}, t_k)} \phi_{k-1},
	\end{equation*}
	then we refer it as an $U^p$-atom. Then, we can define $U^p(\mathbb{R}, \mathcal{H})$ spaces as follows
	\begin{align*}
		u = \sum_{j=1}^{\infty} \lambda_{j} a_j,\qquad  \text{for}\hspace{1ex} a_j \hspace{1ex} \text{are}\hspace{1ex}   U^p\hspace{1ex} \text{-atoms}\  ,\quad \{\lambda_j\}_j \in \ell^1,\quad \|u\|_{U^p}<\infty,
	\end{align*}
	where
	\begin{align*}
		\|u\|_{U^p} : = \inf \bigg\{\sum_{j=1}^{\infty} |\lambda_j| : u = \sum_{j=1}^\infty \lambda_j a_j,\  \lambda_j \in \mathbb{C}\ \text{and}\hspace{1ex} a_j \hspace{1ex} \text{an}\hspace{1ex} U^p\ \text{atom}\bigg\}.
	\end{align*}
\end{definition}

\begin{definition}[$V^p$ spaces \cite{HerrT32011}]
	Let $1\leq p < \infty$ and functions $v:\R\rightarrow \mathcal{H}$. We also further assume that $v(\infty):=0$ and $v(-\infty) := \lim_{t\to -\infty} v(t)$ exists. Then, we define $V^p(\R,\mathcal{H})$ space such that
	\begin{align*}
		\|v\|_{V^p} : = \sup_{\{t_k\}_{k=0}^K\in \mathcal{Z}} \bigg(\sum_{k=1}^K \|v(t_k)-v(t_{k-1})\|^p_\mathcal{H}\bigg)^\frac{1}{p},
	\end{align*} 
	is empty.
	Likewise, we denote $V_{-}^p$ the closed subspace of all $v \in V^p$ with $\lim_{t\to -\infty}v(t) = 0$. We also denote $V_{-, rc}^p$ means all right-continuous $V_{-}^p$ functions.
\end{definition}

\begin{proposition}[Embeding properties \cite{HerrT32011}]\label{fml-pre-emb}
	For $1\leq p \leq q < \infty$, we have
	\begin{equation}
		U^p(\mathbb{R}, \mathcal{H}) \hookrightarrow U^q(\mathbb{R}, \mathcal{H}) \hookrightarrow  L^{\infty}(\mathbb{R},\mathcal{H}),
	\end{equation}
	and
	\begin{equation}
		U^p(\mathbb{R}, \mathcal{H}) \hookrightarrow V^p_{-,rc} (\mathbb{R}, \mathcal{H}) \hookrightarrow U^q(\mathbb{R}, \mathcal{H}).
	\end{equation}
\end{proposition}
 
\begin{definition}[$U^p_{\Delta}H^s$ and $V^p_{\Delta}H^s$ spaces \cite{HerrT32011}]
	For $s\in \mathbb{R}$, we define the $U^p_{\Delta}H^s$ as the space of all functions $u : \mathbb{R}\to H^s(\mathbb{T}^6)$ such that
	\begin{equation*}
		\|u\|_{U^p_{\Delta} H^s} := \|e^{-it\Delta}u(t)\|_{U^p(\mathbb{R},H^s)}.
	\end{equation*}
	Similarly, we can define the $V^p_{\Delta}H^s$ as the space of all functions $u : \mathbb{R}\to H^s(\mathbb{T}^6)$ such that
	\begin{equation*}
		\|u\|_{V^p_{\Delta} H^s} := \|e^{-it\Delta}u(t)\|_{V^p(\mathbb{R},H^s)}.
	\end{equation*} 
\end{definition}

We need the following embedding properties of the $U^p$ and $V^p$ spaces.
\begin{proposition}[Embedding, Proposition 2.2, Proposition 2.4 and Corollary 2.6 in \cite{HadacKp22009}]\label{prop-pre-Embed}
	For $1\leq p \leq q < \infty$,
	\begin{align*}
		U_{\Delta}^p(\R, L^2) \hookrightarrow V_{\Delta}^p (\R, L^2) \hookrightarrow U_{\Delta}^q(\R, L^2) \hookrightarrow  L^{\infty}(\R,L^2) .
	\end{align*}
\end{proposition}

\begin{proposition}[Duality, Theorem 2.8 in \cite{HadacKp22009}]\label{prop-pre-dual}
	Let $DU_\Delta^p$ be the space of functions
	\begin{align*}
		DU_\Delta^p = \{(i\partial_t + \Delta) u : u \in U_\Delta^p \},
	\end{align*}
	and the $DU_\Delta^p = (V_\Delta^{p' })^{\ast}$, with $\frac{1}{p} + \frac{1}{p' } = 1$. Then for any interval $0 \in J$
	\begin{align*}
		\norm{\int_0^t e^{i(t-t' )\Delta}F(u)(t' )\, dt' }_{U_\Delta^p (J\times \R^d)} \lesssim \sup\bracket{\int_J \inner{ v, F}  \,  dt: \norm{v}_{V_\Delta^{p' } }=1 }.
	\end{align*}
\end{proposition}
We also require the following property, which can transfer the map from $L^2$ to $U_\Delta^p$. 
\begin{proposition}[Transfer Principle, Proposition 2.19 in \cite{HadacKp22009}]\label{prop-pre-transpri}
	Let 
	\begin{align*}
		T_0 : L^2 \times \dots \times L^2 \to L_{loc}^{1}
	\end{align*}
	be an m-linear operator. Assume that for some $1\leq p,\ q \leq \infty$
	\begin{align*}
		\norm{ T_0 (e^{it\Delta} \phi_1,\dots, e^{it\Delta} \phi_m)}_{L^p(\R, L_x^q)} \lesssim \prod_{i=1}^m\norm{\phi_i}_{L^2}.
	\end{align*}
	Then, there exists an extension $T : U_{\Delta}^p \times \dots \times U_{\Delta}^p \to L^p(\R, L_x^q)$ satisfying
	\begin{align*}
		\norm{T(u_1, \dots, u_m)}_{L^p(\R, L_x^q)} \lesssim \prod_{i=1}^m \norm{u_i}_{U_{\Delta}^p},
	\end{align*}
	and  such that $T(u_1, \dots , u_m) (t, \cdot) = T_0 (u_1(t), \dots , u_m(t))(\cdot)$, a.e.
\end{proposition}

Now we define the following two function spaces which will be crucial in establishing the local and global well-posedness.
\begin{definition}[Scalar $X^s$ and $Y^s$ spaces in \cite{HerrT32011}]\label{scdef-pre-Xs-Ys}
	For $s\in \mathbb{R}$, we define $X^s$ and $Y^s$ the spaces of all functions $ u : \mathbb{R}\to H^s(\mathbb{T}^6)$ such that for every $n\in \mathbb{Z}^6$, the map $t\mapsto e^{it|n|^2}\widehat{u(t)}(n)$ is in $U^2(\mathbb{R}, \mathbb{C})$ and $V_{rc}^2(\mathbb{R}, \mathbb{C})$, respectively, with the norm
	\begin{align*}
		\norm{u}^2_{X^s} : = \sum_{\xi \in \Z^d} \inner{\xi}^{2s} \norm{\widehat{e^{-it\Delta}u(t)}(\xi)}_{U^2}^2,\\
		\norm{u}^2_{Y^s} : = \sum_{\xi \in \Z^d} \inner{\xi}^{2s} \norm{\widehat{e^{-it\Delta}u(t)}(\xi)}_{V^2}^2.
	\end{align*}
	For intervals $I \subset \R$, we define $X^s (I)$, $s \in \R$ as the restriction norms of $X^s$; thus
	\begin{align*}
		X^s(I) : = \left\{ u \in C(I,H^s) : \norm{u}_{X^s (I)} : = \sup_{\substack{J\subseteq\R\\|J|\leq1}}  \inf_{v  |_J (t) = u |_J (t)} \norm{v}_{{X}^s} < \infty\right\} .
	\end{align*}
	We can define the spaces $Y^s(I)$ in a similar way. Therefore, we have the embedding relation
	\begin{equation}\label{fml-pre-Embed-XY}
		U^2_{\Delta}H^s \hookrightarrow X^s \hookrightarrow Y^s \hookrightarrow V^2_{\Delta}H^s.
	\end{equation} 
	
\end{definition}
\begin{remark}
       For convenience, we always call $u$ is adapted to $e^{it\Delta}$ if $u$ satisfies \eqref{eq:nls}, or other nonlinear terms. So these definitions will change with the structure of equations, especially the linear part, we will see this more clearly in studying $\eqref{eq:system1}$.
\end{remark}
Applying \eqref{fml-pre-Embed-XY} to every Fourier coefficient gives
\begin{equation}\label{eq:XY-energy}
 X^s(I)\hookrightarrow Y^s(I)
 \hookrightarrow L_t^\infty H_x^s(I\times\T^3).
\end{equation}
Naturally, we need to consider $X^s$ and $Y^s$ for $\bfu =(u_1,u_2)$, which is similar to scalar spaces.
\begin{definition}[Vectorial $X^s$ and $Y^s$ spaces]\label{def-pre-Xs-Ys}
	For $s\in \mathbb{R}$, we define $X^s$ and $Y^s$ the spaces of all functions $\bfu=(u_1,u_2) : \mathbb{R}\to H^s(\mathbb{T}^6)$ such that for every $n\in \mathbb{Z}^6$, the map $t\mapsto e^{it|n|^2}\widehat{u_{i}(t)}(n)$, where $i=1,2$, is in $U^2(\mathbb{R}, \mathbb{C})$ and $V_{rc}^2(\mathbb{R}, \mathbb{C})$, respectively, with the norm
	\begin{align*}
 \|\mathbf u\|_{X^s}^2
 :=
 \sum_{\xi\in\mathbb Z^6}\langle\xi\rangle^{2s}
 \left(
 \bigl\|e^{it|\xi|^2}\widehat u_1(t,\xi)\bigr\|_{U^2}^2
 +
 \bigl\|e^{\frac{i}{2}t|\xi|^2}
          \widehat u_2(t,\xi)\bigr\|_{U^2}^2
 \right).\\
 \|\mathbf u\|_{Y^s}^2
 :=
 \sum_{\xi\in\mathbb Z^6}\langle\xi\rangle^{2s}
 \left(
 \bigl\|e^{it|\xi|^2}\widehat u_1(t,\xi)\bigr\|_{V^2}^2
 +
 \bigl\|e^{\frac{i}{2}t|\xi|^2}
          \widehat u_2(t,\xi)\bigr\|_{V^2}^2
 \right).
	\end{align*}

	For intervals $I \subset \R$, we define $X^s (I)$, $s \in \R$ as the restriction norms of $X^s$; thus
	\begin{align*}
		X^s(I) : = \left\{ \bfu \in C(I,H^s) : \norm{\bfu}_{X^s (I)} : = \sup_{\substack{J\subseteq\R\\|J|\leq1}}  \inf_{\mathbf{v} |_J (t) = \bfu |_J (t)} \norm{\mathbf{v}}_{{X}^s} < \infty\right\} .
	\end{align*}
\end{definition}
\begin{proposition}[\cite{HadacKp22009}]\label{prop-pre-Xs-Hs}
	Suppose $\bfu :=  e^{it\Delta}\boldsymbol{\phi}$, then for any $T>0$ we have
	\begin{equation*}
		\|\bfu\|_{X^s([0,T])} \leq \|\boldsymbol{\phi}\|_{H^s}.
	\end{equation*}
\end{proposition}
We denote the norm $N^s(I)$ by the $X^s$-norm of the Duhamel term.
\begin{definition}[Nonlinear norm $N(I)$]\label{def-pre-Nnorm}
	Let $I=[0, T]$, we define the nonlinear norm $N^s(I)$ as
	\begin{equation*}
		\|\mathbf{f}\|_{N^s(I)} := \left\| \int_0^t e^{i(t-t')\Delta} \mathbf{f}(t') dt'\right\|_{X^{s}(I)}.
	\end{equation*}
\end{definition}

We recall following Duhamel duality.
\begin{proposition}\label{prop:duality}
Let $s \ge 0$, and $I=[t_0,b)$, then 
\begin{equation}\label{eq:duality}
 \|\mathbf{f}\|_{N^s(I)}
 \lesssim
 \sup_{\|\mathbf{v}\|_{Y^{-s}(I)}=1}
 \left|\int_I\int_{\mathbb{T}^6}\mathbf{f}(t,x)\cdot \overline{\mathbf{v}(t,x)}\dd x\dd t\right|.
\end{equation}
\end{proposition}
\begin{definition}[$Z$ norm]\label{Z}\label{pre:Z}
We define
\begin{equation}\label{pre:Zformula}
 \|\bfu \|_{Z(I)}=
 \sup_{J\subset I,\ |J|\leq1}
 \left(\sum_N N^{2}\|P_N\bfu\|_{L_{t,x}^3(J\times\mathbb T^6)}^3\right)^{1/3}.
\end{equation}
\end{definition}
\begin{remark}
	Above definitions and propositions is same hold for scalar function $u$.
\end{remark}
\subsection{Linear Strichartz estimates for frequency-localized data}
In this part, we recall some linear Strichartz estimates in $\mathbb{T}^6$.
\begin{proposition}[Strichartz estimates \cite{Bour1993,BD}]
	\label{prop-pre-Stri}
	If $p>\frac{8}{3}$, then
	\begin{equation*}
		\|P_N e^{it\Delta} f\|_{L^p_{t,x}([-1,1]\times\T^6)} \lesssim_p
		N^{3-\frac{8}{p}} \|f\|_{L^2_x}
	\end{equation*}
	and
	\begin{equation*}
		\|P_C e^{it\Delta} f\|_{L^p_{t,x}([-1,1]\times\T^6)} \lesssim_p
		N^{3-\frac{8}{p}} \|f\|_{L^2_x}
	\end{equation*}
	where $C$ is a cube of side length N and $f\in L^2(\T^6)$.
\end{proposition}

By \Cref{prop-pre-transpri} and \Cref{prop-pre-Stri}, we obtain the following corollary:
\begin{corollary}\label{cor-strichartz}
	If $p>\frac{8}{3}$, for any $v\in U^p_\Delta([-1, 1])$,
	\[
	\|P_{\le N} v\|_{L_{t,x}^p([-1,1]\times\T^6)} \lesssim_p N^{3-\frac{8}{p}}
	\|P_{\le N}v\|_{U^p_\Delta L^2}\lesssim N^{3-\frac{8}{p}}\|P_{\le N}v\|_{Y^0([0,1])},
	\]
	and
	\[
	\|P_C v\|_{L^p([-1,1]\times\T^3)} \lesssim_p N^{3-\frac{8}{p}}
	\|v\|_{U^p_\Delta([-1, 1])},
	\]
	where $C$ is a cube of side length $N$.
\end{corollary}
\begin{remark}\label{Yembedding}
      By \eqref{pre:Zformula} and \Cref{cor-strichartz}, we have 
      \[\|u\|_{Z(I)}\lesssim \|u\|_{X^1(I)}.\]
\end{remark}
As a corollary, we have following rough bilinear estimates.
\begin{corollary}\label{cubeproducts}\label{loc:cubes}
Let $N_1\gtrsim  N_2\gtrsim1$. Then 
\begin{align}
 \|P_{N_1}u_1P_{N_2}u_2\|_{L_{t,x}^2(I)}
 &\lesssim N_2^2\|P_{N_1}u_1\|_{Y^0(I)}\|P_{N_2}u_2\|_{Y^0(I)},\label{cubehl}\\
 \|P_{\lesssim N_0}(P_{N_1}u_1P_{N_2}u_2)\|_{L_{t,x}^2(I)}
 &\lesssim N_0^2\|P_{N_1}u_1\|_{Y^0(I)}\|P_{N_2}u_2\|_{Y^0(I)},
 \quad N_1\sim N_2\geq N_0.\label{cubehh}
\end{align}
\end{corollary}
\begin{proof}
Partition the $N_1$ annulus into cubes $C$ of side comparable to $N_2$, so the Fourier supports of
$(P_Cu_1)(P_{N_2}u_2)$ have bounded overlap.  Plancherel, H\"older and
\Cref{cor-strichartz} gives
\begin{align*}
 \|P_{N_1}u_1P_{N_2}u_2\|_2^2
 \leq\sum_C\|P_Cu_1\|_4^2\|P_{N_2}u_2\|_4^2
 \lesssim N_2^4
 \sum_C\|P_Cu_1\|_{Y^0}^2\,
 \|P_{N_2}u_2\|_{Y^0}^2 \lesssim \|P_{N_1}u_1\|_{Y^0}^2  \|P_{N_2}u_2\|_{Y^0}^2.
\end{align*}
Taking square roots proves \eqref{cubehl}.

For \eqref{cubehh}, similarly, partition the $N_1$ and $N_2$ annulus into cubes $C$ and $C'$ of side comparable to $N_0$,
\[
 \|P_{\lesssim N_0}(P_{N_1}u_1P_{N_2}u_2)\|_2
 \lesssim N_0^2\sum_{C\sim C'}
 \|P_C P_{N_1}u_1\|_{Y^0}\|P_{C'}P_{N_2}u_2\|_{Y^0},
\]
and Cauchy--Schwarz gives the result. 
\end{proof}
We recall smooth projection introduced in \cite{HadacKp22009} to handle KP-II equation
\begin{definition}\label{mdef}
Let $S=ia\Delta$, where $a\in\mathbb R$.
Choose an even function $\chi\in C_c^\infty((-2,2))$
such that $\chi=1$ on $[-1,1]$, and set
\[
 \psi(\tau)=\chi(\tau)-\chi(2\tau),
 \qquad
 \psi_M(\tau)=\psi(M^{-1}\tau).
\]
For $M>0$, define Fourier multipliers
$Q_M$ and $Q_{<M}$ by
\[
 \mathcal F_t(Q_Mv)(\tau)
 =\psi_M(\tau)\mathcal F_tv(\tau),
 \qquad
 \mathcal F_t(Q_{<M}v)(\tau)
 =\chi(2\tau/M)\mathcal F_tv(\tau),
\]
and $Q_{\ge M}=I-Q_{<M}$.
The modulation projections associated with $e^{tS}$
are defined by
\[
 Q_M^S=e^{tS}Q_Me^{-tS},
 \qquad
 Q_{<M}^S=e^{tS}Q_{<M}e^{-tS},
 \qquad
 Q_{\ge M}^S=e^{tS}Q_{\ge M}e^{-tS}.
\]
Equivalently,
\begin{align*}
 \mathcal F_{t,x}(Q_M^Su)(\tau,\xi)
 &=
 \psi_M(\tau+a|\xi|^2)
 \mathcal F_{t,x}u(\tau,\xi),\\
 \mathcal F_{t,x}(Q_{<M}^Su)(\tau,\xi)
 &=
 \chi\!\left(\frac{2(\tau+a|\xi|^2)}{M}\right)
 \mathcal F_{t,x}u(\tau,\xi),
\end{align*}
and $Q_{\ge M}^S=I-Q_{<M}^S$.
\end{definition}
\begin{remark}
    For the first and second components of the system \eqref{eq:system1},
we take $S=i\Delta$ and $S=\frac{i}{2}\Delta$,
respectively.
\end{remark}
We consider a high-frequency estimate of this projection.
\begin{proposition}\label{mM}
For every $M>0$ and $u\in Y^0$, one has
\begin{align}\label{2.9}
 \|Q_{\ge M}^S u\|_{L^2(\mathbb R\times\mathbb T^6)}
 \lesssim M^{-1/2}\|u\|_{Y^0}.
\end{align}
\end{proposition}
\begin{proof}
Let $v=e^{-tS}u$. By the definitions, $e^{-tS}Q_{\ge M}^Su=Q_{\ge M}v$ and 
 $e^{-tS}Q_{<M}^Su=Q_{<M}v$.

The Besov embedding
$V^2\hookrightarrow\dot B^{1/2}_{2,\infty}$
of~\cite{HadacKp22009}
gives, for every $\xi\in\mathbb Z^6$ and $j\ge0$,
\[
 \|Q_{2^jM}\widehat v(\cdot,\xi)\|_{L^2_t}
 \lesssim (2^jM)^{-1/2}
          \|\widehat v(\cdot,\xi)\|_{V^2}.
\]
So 
\begin{align*}
 \|Q_{\ge M}\widehat v(\cdot,\xi)\|_{L^2_t}
 \le
 \sum_{j=0}^{\infty}
 \|Q_{2^jM}\widehat v(\cdot,\xi)\|_{L^2_t}
 \lesssim
 \sum_{j=0}^{\infty}(2^jM)^{-1/2}
 \|\widehat v(\cdot,\xi)\|_{V^2}
 \lesssim
 M^{-1/2}\|\widehat v(\cdot,\xi)\|_{V^2}.
\end{align*}
By unitarity of $e^{tS}$ and Plancherel,
\begin{align*}
 \|Q_{\ge M}^Su\|_{L^2_{t,x}}^2
 =\|Q_{\ge M}v\|_{L^2_{t,x}}^2
=\sum_{\xi\in\mathbb Z^6}
   \|Q_{\ge M}\widehat v(\cdot,\xi)\|_{L^2_t}^2
 \lesssim
 M^{-1}\sum_{\xi\in\mathbb Z^6}
 \|\widehat v(\cdot,\xi)\|_{V^2}^2
 =M^{-1}\|u\|_{Y^0}^2.
\end{align*}
Taking square roots proves \eqref{2.9}.
\end{proof}
\begin{remark}
    Notice that we obtain a gain in the high-frequency estimates. It's important for us.
\end{remark}
\section{Local well-posedness and stability theory}\label{se3}
First, we proving multilinear estimates to handle nonlinear term.
\subsection{Multilinear estimates}
Before prove multilinear estimate, we get a bilinear estimate first.
\begin{lemma}
      Let $I\subset\mathbb R$ be an interval with $|I|\le1$. Fix $M>0$ and, for each $j$, let
\[
	 Q_j\in\{\mathrm{Id},Q_{<M}^{S_j},Q_{\ge M}^{S_j}\}.
\]
For distinct $j,k\in\{0,1,2\}$, $f_j, f_k \in Y^0(I)$ and we extend them by zero outside $I$. For convenience, we also donate it by $f_j, f_k$, then we have 
\begin{equation}\label{two}
	 \|(Q_jf_j)(Q_kf_k)\|_{L^2_{t,x}(\mathbb R\times\mathbb T^6)}
 \lesssim \min\{N_j,N_k\}^2
 \|f_j\|_{Y^0}\|f_k\|_{Y^0},
\end{equation}
and, when $N_j\sim N_k\ge N\ge1$,
\begin{equation}\label{simtwo}
	 \|P_{\lesssim N}((Q_jf_j)(Q_kf_k))
       \|_{L^2_{t,x}(\mathbb R\times\mathbb T^6)}
 \lesssim N^2
 \|f_j\|_{Y^0}\|f_k\|_{Y^0}.
\end{equation}
\end{lemma}
\begin{proof} 
First, we consider the case that $Q_j= Q^{S_j}_{<M}$ and $Q_k= Q^{S_k}_{<M}$. We set $K_M$ is the convolution kernel of $Q_{<M}$ from
\Cref{mdef}, then $K_M(r)=MK_1(Mr)$ and
$\|K_M\|_{L^1}=\|K_1\|_{L^1}<\infty$,
so
\[
 Q_{<M}^{S_j}f_j(t)
 =\int_{\mathbb R}K_M(r)e^{rS_j}f_j(t-r)\,dr.
\]
Notice that $e^{rS_j}f_j(t-r)$
is supported in
$I+r$, and it satisfies $
 \|e^{rS_j}f_j(\cdot-r)\|_{Y^0}
 =\|f_j\|_{Y^0}$. 
Consequently, $(e^{rS_j}f_j(\cdot-r))(e^{sS_k}f_k(\cdot-s))$ is supported in
$(I+r)\cap(I+s)$, an interval of length at most one.
So by \eqref{cubehl},
\[
 \|(e^{rS_j}f_j(\cdot-r))(e^{sS_k}f_k(\cdot-s))\|_{L^2_{t,x}}
 \lesssim\min\{N_j,N_k\}^2
 \|f_j\|_{Y^0}\|f_k\|_{Y^0}.
\]
Minkowski's inequality therefore gives
\begin{align*}
 \|(Q_{<M}^{S_j}f_j)(Q_{<M}^{S_k}f_k)\|_{L^2_{t,x}}
 &\le\int_{\mathbb R^2}|K_M(r)K_M(s)|
 \|(e^{rS_j}f_j(\cdot-r))(e^{sS_k}f_k(\cdot-s))\|_{L^2_{t,x}}
 \,dr\,ds\\
 &\lesssim\|K_1\|_{L^1}^2\min\{N_j,N_k\}^2
 \|f_j\|_{Y^0}\|f_k\|_{Y^0}.
\end{align*}
Suppose that
$N_j\sim N_k\ge N\ge1$. Similarly, Minkowski's inequality and \eqref{cubehh} yields
\begin{align*}
 &\quad \left\|P_{\lesssim N}
 \bigl((Q_{<M}^{S_j}f_j)(Q_{<M}^{S_k}f_k)\bigr)
 \right\|_{L^2_{t,x}}\\
 &\le\int_{\mathbb R^2}|K_M(r)K_M(s)|
 \left\|P_{\lesssim N}
 \bigl((e^{rS_j}f_j(\cdot-r))(e^{sS_k}f_k(\cdot-s))\bigr)
 \right\|_{L^2_{t,x}}\,dr\,ds\\
 &\lesssim\|K_M\|_{L^1}^2N^2
 \|f_j\|_{Y^0}\|f_k\|_{Y^0}
 \lesssim N^2\|f_j\|_{Y^0}\|f_k\|_{Y^0}.
\end{align*}
Then we consider the case that $Q_j= Q^{S_j}_{<M}$. We have
\[
 (Q_{<M}^{S_j}f_j)(t)f_k(t)
 =\int_{\mathbb R}K_M(r)
 (e^{rS_j}f_j(t-r))f_k(t)\,dr.
\]
Similarly, $(e^{rS_j}f_j(t-r))f_k(t)$ is supported in
$(I+r)\cap I$. Applying Minkowski's inequality and \eqref{cubehl} gives
\begin{align*}
 \|(Q_{<M}^{S_j}f_j)f_k\|_{L^2_{t,x}}
 &\le\int_{\mathbb R}|K_M(r)|
 \|(e^{rS_j}f_j(\cdot-r))f_k\|_{L^2_{t,x}}\,dr\\
 &\lesssim\|K_M\|_{L^1}\min\{N_j,N_k\}^2
 \|f_j\|_{Y^0}\|f_k\|_{Y^0}\\
 &\lesssim\min\{N_j,N_k\}^2
 \|f_j\|_{Y^0}\|f_k\|_{Y^0}.
\end{align*}
When $N_j\sim N_k\ge N\ge1$, applying \eqref{cubehh} instead gives
\begin{align*}
 \|P_{\lesssim N}((Q_{<M}^{S_j}f_j)f_k)\|_{L^2_{t,x}}
 &\le\int_{\mathbb R}|K_M(r)|
 \|P_{\lesssim N}((e^{rS_j}f_j(\cdot-r))f_k)\|_{L^2_{t,x}}\,dr\\
 &\lesssim\|K_M\|_{L^1}N^2
 \|f_j\|_{Y^0}\|f_k\|_{Y^0}
 \lesssim N^2\|f_j\|_{Y^0}\|f_k\|_{Y^0}.
\end{align*}
If $Q_j=Q_k= Id$, \eqref{two}, \eqref{simtwo} follow directly from \eqref{cubehh} and \eqref{cubehl}.

It remains to treat $Q_j= Q^{S_j}_{\ge M}$ or $Q_k= Q^{S_k}_{\ge M}$. Since
$Q_{\ge M}^{S_j}=\mathrm{Id}-Q_{<M}^{S_j}$, we have
\begin{align*}
 (Q_{\ge M}^{S_j}f_j)f_k
 &=f_jf_k-(Q_{<M}^{S_j}f_j)f_k,\\
 (Q_{\ge M}^{S_j}f_j)(Q_{<M}^{S_k}f_k)
 &=f_j(Q_{<M}^{S_k}f_k)
   -(Q_{<M}^{S_j}f_j)(Q_{<M}^{S_k}f_k),
\end{align*}
and
\begin{align*}
 &(Q_{\ge M}^{S_j}f_j)(Q_{\ge M}^{S_k}f_k)=f_jf_k-(Q_{<M}^{S_j}f_j)f_k
       -f_j(Q_{<M}^{S_k}f_k)
       +(Q_{<M}^{S_j}f_j)(Q_{<M}^{S_k}f_k).
\end{align*}
Every term on the right-hand sides has already been estimated.
So by the triangle inequality, we prove \eqref{two}, \eqref{simtwo} for all choices of $Q_j$ and $Q_k$.
\end{proof}
\begin{proposition}\label{prop:trilinear-frequency-bound}
Let $I\subset\mathbb R$ be an interval with $|I|\le1$, and let
$N_0,N_1,N_2\ge1$ be dyadic numbers with $N_1\ge N_2$.
For $j=0,1,2$, let $u_j$ be a scalar function in $Y^0(I)$ adapted
to $e^{tS_j}$, where $S_j=ia_j\Delta$ and
\[
 (a_0,a_1,a_2)\in
 \left\{
 \left(1,\tfrac12,-1\right),
 \left(1,-1,\tfrac12\right),
 \left(\tfrac12,1,1\right)
 \right\}.
\]
Then
\begin{equation}\label{eq:trilinear-frequency-bound}
\begin{aligned}
 &\quad \left|\int_{I\times\mathbb T^6}
   (P_{N_1}u_1)(P_{N_2}u_2)
   \overline{P_{N_0}u_0}\,dx\,dt\right|\\
 &\lesssim
 \left(N_0^{-1}\|P_{N_0}u_0\|_{Y^0(I)}\right)
 \prod_{j=1}^2\left(N_j\|P_{N_j}u_j\|_{Y^0(I)}\right)\times
 \begin{cases}
  1, & N_0\sim N_1\sim N_2,\\
  N_2/N_1, & N_0\sim N_1\gg N_2,\\
  (N_0/N_1)^3, & N_1\sim N_2\gg N_0.
 \end{cases}
\end{aligned}
\end{equation}
\end{proposition}

\begin{proof}
If $N_0\sim N_1\sim N_2$, H\"older's inequality and \Cref{cor-strichartz} yields
\begin{align*}
 \left|\int_{I\times\mathbb T^6}
 (P_{N_1}u_1)(P_{N_2}u_2)
 \overline{P_{N_0}u_0}\,dx\,dt\right|
 &\le\prod_{j=0}^2\|P_{N_j}u_j\|_{L^3_{t,x}(I\times\mathbb T^6)}\lesssim\prod_{j=0}^2
 N_j^{1/3}\|P_{N_j}u_j\|_{Y^0(I)}\\
 &\lesssim\frac{N_1N_2}{N_0}
 \prod_{j=0}^2\|P_{N_j}u_j\|_{Y^0(I)}.
\end{align*}
This proves the first case.

For other cases, we extend each $P_{N_j}u_j$ by
zero outside $I$, and denote this extension by $f_j$. By definition of $Y^s$, we obtain $
 \|f_j\|_{Y^0}
 \sim\|P_{N_j}u_j\|_{Y^0(I)}$.
The Fourier support of $f_j$ remains in the support of
$P_{N_j}$, and
\[
 \int_{I\times\mathbb T^6}
 (P_{N_1}u_1)(P_{N_2}u_2)\overline{P_{N_0}u_0}\,dx\,dt
 =\int_{\mathbb R\times\mathbb T^6}f_1f_2\overline{f_0}\,dx\,dt.
\]

We next use the resonance relation. In frequency space, for $\hat{f_i}(\tau_i,\xi_i)$, $i=0,1,2$, one has
$\xi_0=\xi_1+\xi_2$ and $\tau_0=\tau_1+\tau_2$.
So 
\[
\begin{aligned}
(\tau_0+a_0|\xi_0|^2)
  -(\tau_1+a_1|\xi_1|^2)
  -(\tau_2+a_2|\xi_2|^2)
 =
 \begin{cases}
 \frac12|\xi_1+2\xi_2|^2,
       &(a_0,a_1,a_2)=(1,\tfrac12,-1),\\
 \frac12|2\xi_1+\xi_2|^2,
       &(a_0,a_1,a_2)=(1,-1,\tfrac12),\\
 -\frac12|\xi_1-\xi_2|^2,
       &(a_0,a_1,a_2)=(\tfrac12,1,1).
 \end{cases}
\end{aligned}
\]

If $N_0\sim N_1\gg N_2$, by triangle inequality, we have 
\[
|(\tau_0+a_0|\xi_0|^2)
  -(\tau_1+a_1|\xi_1|^2)
  -(\tau_2+a_2|\xi_2|^2)|
 \gtrsim N_1^2.
\]
Choose $M \sim c N_1^2$, and decompose 
$f_j=Q_{<M}^{S_j}f_j+Q_{\ge M}^{S_j}f_j$.
We notice that by definition, 
if $|\tau_j+a_j|\xi_j|^2|\le M$, then $|(\tau_0+a_0|\xi_0|^2)
  -(\tau_1+a_1|\xi_1|^2)
  -(\tau_2+a_2|\xi_2|^2)| \lesssim M$.
This implies that
\[
 \int_{\mathbb R\times\mathbb T^6}
 (Q_{<M}^{S_1}f_1)(Q_{<M}^{S_2}f_2)
 \overline{Q_{<M}^{S_0}f_0}\,dx\,dt=0.
\]
when $M \sim c N_1^2$ for a suitable $c$.

So each of the seven remaining terms has at least one $Q^{S_j}_{\ge M}f_j$ for some $j$. By \Cref{mM}
\begin{equation}\label{high}
     \|Q_{\ge M}^{S_j}f_j\|_{L^2_{t,x}(\mathbb R\times\mathbb T^6)}
 \lesssim M^{-1/2}\|f_j\|_{Y^0}
 \lesssim N_1^{-1}\|f_j\|_{Y^0}.
\end{equation}

If $j$ is $0$ or $1$, apply Cauchy--Schwarz to
the $Q^{S_j}_{\ge M}f_j$ factor and the other two factors, and then
use \eqref{two} and \eqref{high}. These give, for example,
\[
 \left|\int_{\mathbb R\times\mathbb T^6}
 Q^{S_1}_{\le M}Q^{S_2}_{\le M}\overline{Q^{S_0}_{\ge M}}\,dx\,dt\right|
 \lesssim\frac{N_2^2}{N_1}
 \prod_{k=0}^2\|f_k\|_{Y^0}.
\]
If $j$ is $2$, by \eqref{simtwo}, for example, we have 
\begin{align*}
 \left|\int_{\mathbb R\times\mathbb T^6}
 Q^{S_1}_{\le M} Q^{S_2}_{\ge M}\overline{Q^{S_0}_{\le M}}\,dx\,dt\right|
 &\le\|Q^{S_2}_{\ge M}\|_{L^2_{t,x}}
       \|P_{\lesssim N_2}(Q^{S_1}_{\le M}\overline{Q^{S_0}_{\le M}})\|_{L^2_{t,x}}\lesssim\frac{N_2^2}{N_1}
       \prod_{k=0}^2\|f_k\|_{Y^0}.
\end{align*}
Since $N_0\sim N_1$,
$N_2^2/N_1\sim (N_1N_2/N_0)(N_2/N_1)$, as required.

If $N_1\sim N_2\gg N_0$, using $\xi_0=\xi_1+\xi_2$, we have
\[
 \xi_1+2\xi_2=\xi_0+\xi_2,\qquad
 2\xi_1+\xi_2=\xi_0+\xi_1,\qquad
 \xi_1-\xi_2=2\xi_1-\xi_0.
\]
Since $|\xi_0|\ll|\xi_1|\sim|\xi_2|$, the triangle
inequality and the three identities above give
\[
 |(\tau_0+a_0|\xi_0|^2)
  -(\tau_1+a_1|\xi_1|^2)
  -(\tau_2+a_2|\xi_2|^2)|
 \gtrsim N_1^2.
\]
Choose $M=cN_1^2$ with suitable $c>0$, similarly
\[
 \int_{\mathbb R\times\mathbb T^6}
 (Q_{<M}^{S_1}f_1)(Q_{<M}^{S_2}f_2)
 \overline{Q_{<M}^{S_0}f_0}\,dx\,dt=0.
\]

Also, each of the seven remaining terms has at least one $Q^{S_j}_{\ge M}f_j$ for some $j$. 
By \Cref{mM},
\begin{equation}\label{high2}
     \|Q_{\ge M}^{S_j}f_j\|_{L^2_{t,x}(\mathbb R\times\mathbb T^6)}
 \lesssim M^{-1/2}\|f_j\|_{Y^0}
 \lesssim N_1^{-1}\|f_j\|_{Y^0}.
\end{equation}
If $j$ is $1$ or $2$, applying Cauchy--Schwarz, \eqref{two} and \eqref{high2} give, for example,
\begin{align*}
 \left|\int_{\mathbb R\times\mathbb T^6}
 (Q_{\ge M}^{S_1}f_1)(Q_{<M}^{S_2}f_2)
 \overline{Q_{<M}^{S_0}f_0}\,dx\,dt\right| \le
 \|Q_{\ge M}^{S_1}f_1\|_{L^2_{t,x}}
 \|(Q_{<M}^{S_2}f_2)
   \overline{Q_{<M}^{S_0}f_0}\|_{L^2_{t,x}} \lesssim\frac{N_0^2}{N_1}
 \prod_{k=0}^2\|f_k\|_{Y^0}.
\end{align*}
If $j$ is $0$, similarly, we have, for example,
\[
\begin{aligned}
 \left|\int_{\mathbb R\times\mathbb T^6}
 (Q_{<M}^{S_1}f_1)(Q_{<M}^{S_2}f_2)
 \overline{Q_{\ge M}^{S_0}f_0}\,dx\,dt\right|
 &\le
 \|Q_{\ge M}^{S_0}f_0\|_{L^2_{t,x}}
 \left\|P_{\lesssim N_0}
 \bigl((Q_{<M}^{S_1}f_1)(Q_{<M}^{S_2}f_2)\bigr)
 \right\|_{L^2_{t,x}} \\
 &\lesssim\frac{N_0^2}{N_1}
 \prod_{k=0}^2\|f_k\|_{Y^0}.
\end{aligned}
\]
Finally, since $N_1\sim N_2$, we have
 $N_0^2/{N_1}
 \sim({N_1N_2}/{N_0})
       ({N_0}/{N_1})^3$.
\end{proof}
We define
\begin{equation}\label{eq:z-prime-norm}
 \|u\|_{Z'(I)}
 :=\|u\|_{Z(I)}^{1/2}\|u\|_{X^1(I)}^{1/2}.
\end{equation}

\begin{proposition}\label{lem:trilinear-l3-bound}
Under the hypotheses of
\Cref{prop:trilinear-frequency-bound}, one has
\begin{equation}\label{eq:trilinear-l3-bound}
\begin{aligned}
 &\quad \left|\int_{I\times\mathbb T^6}
 (P_{N_1}u_1)(P_{N_2}u_2)
 \overline{P_{N_0}u_0}\,dx\,dt\right|\\
 &\lesssim
 \left(N_0^{-1}\|P_{N_0}u_0\|_{Y^0(I)}\right)
 \prod_{j=1}^2
 \left(N_j^{2/3}\|P_{N_j}u_j\|_{L^3_{t,x}(I)}\right) \times 
 \begin{cases}
  1, & N_0\sim N_1\sim N_2,\\
  (N_1/N_2)^{2/3}, & N_0\sim N_1\gg N_2,\\
  (N_0/N_1)^{4/3}, & N_1\sim N_2\gg N_0.
 \end{cases}.
\end{aligned}
\end{equation}
\end{proposition}

\begin{proof}
Since $|I|\le1$, \Cref{cor-strichartz} with $p=3$, applied to the
propagator associated with $u_0$, gives
\[
 \|P_{N_0}u_0\|_{L^3_{t,x}(I)}
 \lesssim N_0^{1/3}\|P_{N_0}u_0\|_{Y^0(I)}.
\]
H\"older's inequality therefore yields
\begin{align*}
 &\quad \left|\int_{I\times\mathbb T^6}
 (P_{N_1}u_1)(P_{N_2}u_2)
 \overline{P_{N_0}u_0}\,dx\,dt\right|\le\prod_{j=0}^2
 \|P_{N_j}u_j\|_{L^3_{t,x}(I)}\\
 &\lesssim
 N_0^{1/3}\|P_{N_0}u_0\|_{Y^0(I)}
 \prod_{j=1}^2\|P_{N_j}u_j\|_{L^3_{t,x}(I)}\\
 &=
 \frac{N_0^{4/3}}{(N_1N_2)^{2/3}}
 \left(N_0^{-1}\|P_{N_0}u_0\|_{Y^0(I)}\right)
 \prod_{j=1}^2
 \left(N_j^{2/3}\|P_{N_j}u_j\|_{L^3_{t,x}(I)}\right).
\end{align*}
Combing with the frequency relations
into $N_0^{4/3}/(N_1N_2)^{2/3}$, we obtain \eqref{eq:trilinear-l3-bound}.
\end{proof}

\begin{remark}\label{rem:trilinear-geometric-mean}
Taking the geometric mean of the bound in
\Cref{prop:trilinear-frequency-bound} and
\Cref{lem:trilinear-l3-bound}, we obtain
\begin{equation}\label{eq:trilinear-geometric-mean}
\begin{aligned}
 &\quad\left|\int_{I\times\mathbb T^6}
 (P_{N_1}u_1)(P_{N_2}u_2)
 \overline{P_{N_0}u_0}\,dx\,dt\right|\\
 &\lesssim
 \left(N_0^{-1}\|P_{N_0}u_0\|_{Y^0(I)}\right)
 \prod_{j=1}^2
 \left[
 N_j\|P_{N_j}u_j\|_{Y^0(I)}
 N_j^{2/3}\|P_{N_j}u_j\|_{L^3_{t,x}(I)}
 \right]^{1/2}\\
 &\quad \times
 \begin{cases}
  1, & N_0\sim N_1\sim N_2,\\
  (N_2/N_1)^{1/6}, & N_0\sim N_1\gg N_2,\\
  (N_0/N_1)^{13/6}, & N_1\sim N_2\gg N_0.
 \end{cases}
\end{aligned}
\end{equation}
\end{remark}
Now, we introduce our nonlinear estimates to establish local theory.
\begin{lemma}\label{lem:nonlinear-estimates}
Let $I=[a,b]$ with $|I|\le1$, and let $t_0\in I$.
For $\bfu=(u_1,u_2)\in X^1(I)$, one has
\begin{equation}\label{eq:quadratic-nonlinear-n}
 \left\|
\int_{t_0}^t e^{i(t-r)\Delta}
             F(\mathbf{u})\,dr
 \right\|_{X^1(I)}
 \lesssim \|\bfu\|_{Z'(I)}^2.
\end{equation}
Moreover, for $\bfu\in X^1(I)$,
\begin{equation}\label{eq:quadratic-system-xs}
 \left\|
\int_{t_0}^t e^{i(t-r)\Delta}
             F(\mathbf{u})\,dr
 \right\|_{X^1(I)}
 \lesssim \|\bfu\|_{X^1(I)}\|\bfu\|_{Z'(I)}.
\end{equation}
\end{lemma}

\begin{proof}
Applying \Cref{prop:duality}, we obtain
\begin{align*}
 \left\|
\int_{t_0}^t e^{i(t-r)\Delta}
             F(\mathbf{u})\,dr
 \right\|_{X^1(I)}
 \lesssim
 \sup_{\|\textbf{v}\|_{Y^{-1}(I)}\le1}
 \left|\int_{I\times\mathbb T^6}
 \bigl(u_2\overline{u_1}\,\overline{v_1}
       +u_1^2\overline{v_2}\bigr)\,dx\,dt\right|.
\end{align*}
Here $\mathbf{v}=(v_1,v_2)$, where $v_1$ is adapted to $e^{it\Delta}$
and $v_2$ to $e^{\frac i2t\Delta}$.

First, we consider
\[\sum_{N_0,N_1,N_2}
 \left|\int_{I\times\mathbb T^6}
 (P_{N_1}u_2)\overline{P_{N_2}u_1}\, 
  \overline{P_{N_0}v_1}\,dx\,dt\right|\]

We assume $N_1 \ge N_2$, if $N_1 \sim N_2 \sim N_0$, applying \eqref{eq:trilinear-geometric-mean} and H\"older's
inequality gives
\begin{align*}
 &\sum_{N_0\sim N_1\sim N_2}
 \left|\int_{I\times\mathbb T^6}
 (P_{N_1}u_2)\overline{P_{N_2}u_1}\,
 \overline{P_{N_0}v_1}\,dx\,dt\right|\\
 &\lesssim
 \left\{\sum_N
 \left(N^{-1}\|P_Nv_1\|_{Y^0(I)}\right)^6
 \right\}^{1/6}\times
 \prod_{j=1}^2
 \left\{\sum_N
 \left[
 N\|P_Nu_j\|_{Y^0(I)}
 N^{2/3}\|P_Nu_j\|_{L^3_{t,x}(I)}
 \right]^{6/5}\right\}^{5/12}\\
 &\lesssim
 \left(\sum_NN^{-2}\|P_Nv_1\|_{Y^0(I)}^2\right)^{1/2}\times
 \prod_{j=1}^2
 \left\{\sum_N
 \left[
 N\|P_Nu_j\|_{Y^0(I)}
 N^{2/3}\|P_Nu_j\|_{L^3_{t,x}(I)}
 \right]^{6/5}\right\}^{5/12}\\
 &\lesssim
 \|v_1\|_{Y^{-1}(I)}
 \|u_1\|_{Z'(I)}\|u_2\|_{Z'(I)}.
\end{align*}
If $N_0 \sim N_1 \gg N_2$, by \eqref{eq:trilinear-geometric-mean} and 
Cauchy--Schwarz, we obtain
\begin{align*}
 &\sum_{N_0\sim N_1\gg N_2}
 \left|\int_{I\times\mathbb T^6}
 (P_{N_1}u_2)\overline{P_{N_2}u_1}\,
 \overline{P_{N_0}v_1}\,dx\,dt\right|\\
 &\lesssim
 \left(\sum_NN^{-2}\|P_Nv_1\|_{Y^0(I)}^2\right)^{1/2}
 \prod_{j=1}^2
 \left\{
 \sum_N
 \left[
 N\|P_Nu_j\|_{Y^0(I)}
 N^{2/3}\|P_Nu_j\|_{L^3_{t,x}(I)}
 \right]^2
 \right\}^{1/4}\\
 &\le
 \left(\sum_NN^{-2}\|P_Nv_1\|_{Y^0(I)}^2\right)^{1/2}
 \prod_{j=1}^2
 \left\{
 \sum_N
 \left[
 N\|P_Nu_j\|_{Y^0(I)}
 N^{2/3}\|P_Nu_j\|_{L^3_{t,x}(I)}
 \right]^{6/5}
 \right\}^{5/12}\\
 &\lesssim
 \|v_1\|_{Y^{-1}(I)}
 \|u_1\|_{Z'(I)}\|u_2\|_{Z'(I)}.
\end{align*}
If $N_1\sim N_2\gg N_0$, applying \eqref{eq:trilinear-geometric-mean} and
Cauchy--Schwarz twice, we obtain
\begin{align*}
&\sum_{N_1\sim N_2\gg N_0}
\left|
\int_{I\times\mathbb T^6}
(P_{N_1}u_2)\overline{P_{N_2}u_1}\,
\overline{P_{N_0}v_1}\,dx\,dt
\right|
\\
&\lesssim
\left(
\sum_N N^{-2}\|P_Nv_1\|_{Y^0(I)}^2
\right)^{1/2}
\prod_{j=1}^{2}
\left\{
\sum_N
\left[
N\|P_Nu_j\|_{Y^0(I)}
N^{2/3}\|P_Nu_j\|_{L^3_{t,x}(I)}
\right]^2
\right\}^{1/4}
\\
&\le
\left(
\sum_N N^{-2}\|P_Nv_1\|_{Y^0(I)}^2
\right)^{1/2}
\prod_{j=1}^{2}
\left\{
\sum_N
\left[
N\|P_Nu_j\|_{Y^0(I)}
N^{2/3}\|P_Nu_j\|_{L^3_{t,x}(I)}
\right]^{6/5}
\right\}^{5/12}
\\
&\lesssim
\|v_1\|_{Y^{-1}(I)}
\|u_1\|_{Z'(I)}\|u_2\|_{Z'(I)}.
\end{align*}
By a similar argument, we can prove 
\[\sum_{N_0,N_1,N_2}
 \left|\int_{I\times\mathbb T^6}
 (P_{N_1}u_1){P_{N_2}u_1}\, 
  \overline{P_{N_0}v_2}\,dx\,dt\right| \lesssim \|v_2\|_{Y^{-1}(I)}
\|u_1\|^2_{Z'(I)}\]
Above all, we obtain
\begin{align*}
 \left|\int_{I\times\mathbb T^6}
 \bigl(u_2\overline{u_1}\,\overline{v_1}
       +u_1^2\overline{v_2}\bigr)\,dx\,dt\right| &\lesssim
 \|v_1\|_{Y^{-1}(I)}
 \|u_1\|_{Z'(I)}\|u_2\|_{Z'(I)}
 +\left\|v_2\right\|_{Y^{-1}(I)}\|u_1\|^2_{Z'(I)}
 \\
 &\lesssim
 \|\textbf{v}\|_{Y^{-1}(I)}\|\bfu\|^2_{Z'(I)}\lesssim  \|\textbf{v}\|_{Y^{-1}(I)}\|\bfu\|_{Z'(I)}\left\|\bfu\right\|_{X^1(I)}.
\end{align*}
Taking the vector dual supremum proves \eqref{eq:quadratic-nonlinear-n} and 
\eqref{eq:quadratic-system-xs}.
\end{proof}
\subsection{Local theory}

We prove local theory by the above nonlinear estimates.

We recall the vector propagator and the nonlinearity from the introduction:
\[
 e^{it\Delta}\textbf{u}_0
 =\begin{pmatrix}e^{it\Delta}u_{1,0}\\ e^{\frac i2t\Delta}u_{2,0}\end{pmatrix},
 \qquad
 F(\textbf{u})
 =\begin{pmatrix}u_2\overline{u_1}\\ u_1^2\end{pmatrix}.
\]
Thus the Duhamel formula for the system is
\[
 \textbf{u}(t)=e^{i(t-t_0)\Delta}\textbf{u}_0
 +i\int_{t_0}^t e^{i(t-r)\Delta}F(\textbf{u}(r))\,dr.
\]
\begin{theorem}
\label{prop:local-theory-h1}
(1) Given $E>0$, there exists $\delta_0=\delta_0(E)>0$ such that,
if $t_0\in I$, $|I|\le1$, and $\textbf{u}_0\in H^1(\mathbb T^6;\mathbb C^2)$ satisfies
\[
 \|\textbf{u}_0\|_{H^1}\le E,
 \qquad
 \|e^{i(t-t_0)\Delta}\textbf{u}_0\|_{Z'(I)}
 \le\delta\le\delta_0,
\]
then the system \eqref{eq:system1} has a unique strong solution
$\textbf{u}\in C(I;H^1)\cap X^1(I)$ with
$\textbf{u}(t_0)=\textbf{u}_0$, and
\begin{equation}
\label{eq:local-h1-correction}
 \|\textbf{u}-e^{i(t-t_0)\Delta}\textbf{u}_0\|_{X^1(I)}
 \lesssim\delta^2.
\end{equation}

(2) Let $\textbf{u}$ be a strong solution on a bounded interval $I$.
If $\|\textbf{u}\|_{Z(I)}<\infty$, then $\textbf{u}$ extends to a
neighborhood of $\overline I$. Moreover, when $|I|\le1$ and $t_0\in I$,
\begin{equation}
\label{eq:local-h1-continuation-bound}
 \|\textbf{u}\|_{X^1(I)}
 \le C\bigl(\|\textbf{u}(t_0)\|_{H^1},\|\textbf{u}\|_{Z(I)}\bigr).
\end{equation}
\end{theorem}

\begin{proof}
(1) Set $\textbf{v}(t)=e^{i(t-t_0)\Delta}\textbf{u}_0$ and write
$\textbf{u}=\textbf{v}+\textbf{w}$. We solve
\[
 \textbf{w}=\Phi(\textbf{w}),
 \qquad
 \Phi(\textbf{w})(t)
 =i\int_{t_0}^t e^{i(t-r)\Delta}
                 \textbf{F}(\textbf{v}(r)+\textbf{w}(r))\,dr.
\]
By \Cref{lem:nonlinear-estimates} and $\|\textbf{v}\|_{Z'(I)}\le\delta$, gives
\begin{align*}
 \|\Phi(\textbf{w})\|_{X^1(I)}
 &\le C\bigl(\delta^2
       +\delta\|\textbf{w}\|_{X^1(I)}
       +\|\textbf{w}\|_{X^1(I)}^2\bigr),\\
 \|\Phi(\textbf{w})-\Phi(\textbf{z})\|_{X^1(I)}
 &\le C\bigl(\delta+\|\textbf{w}\|_{X^1(I)}
                    +\|\textbf{z}\|_{X^1(I)}\bigr)
             \|\textbf{w}-\textbf{z}\|_{X^1(I)}.
\end{align*}
We consider the closed ball
\[
 S= \bigl\{\textbf{w}\in C(I;H^1)\cap X^1(I):
 \textbf{w}(t_0)=0,\quad
 \|\textbf{w}\|_{X^1(I)}\le2C\delta^2\bigr\}.
\]
So choosing 
$\delta_0>0$ sufficiently small that $\Phi$ is a concentraction on ball $S$. The contraction
mapping theorem gives a solution and
\eqref{eq:local-h1-correction}. We mention that uniqueness is easy.

(2) It suffices to consider $I=(0,T)$.
Since $\mathbf u\in X^1(I)$, fix $E>0$ such that
$\sup_{t\in I}\|\mathbf u(t)\|_{H^1}\le E$.
Choose $\varepsilon>0$ sufficiently small and
$T-\frac12<T_1<T$ such that $ \|\mathbf u\|_{Z'(T_1,T)}\le\varepsilon$.
Let
\[
 h(s)=
 \bigl\|e^{i(t-T_1)\Delta}\mathbf u(T_1)
 \bigr\|_{Z'(T_1,T_1+s)}.
\]
For $0<s<T-T_1$, \Cref{lem:nonlinear-estimates} gives
\[
 \bigl\|\mathbf u-e^{i(t-T_1)\Delta}\mathbf u(T_1)
 \bigr\|_{X^1(T_1,T_1+s)}
 \lesssim\varepsilon^2.
\]
Thus, by the definition of $Z'$ and the linear estimate,
\[
 h(s)^2
 \lesssim E\bigl(
 \|\mathbf u\|_{Z(T_1,T_1+s)}+\varepsilon^2
 \bigr)
 \lesssim E(\varepsilon+\varepsilon^2).
\]
Taking $\varepsilon$ sufficiently small gives
$h(s)\le\delta_0(E)/2$.
By continuity, there exists $T_2>T$, with $T_2-T_1<1$,
such that $h(T_2-T_1)<\delta_0(E)$.
Part (1) and uniqueness therefore extend $\mathbf u$
beyond $T$.

Finally, partition $I$ into finitely many intervals $J$
on which $\|\mathbf u\|_{Z(J)} \le \eta$ for some small $\eta$.
\Cref{lem:nonlinear-estimates} gives, for some $t_* \in \overline{J}$
\[
\begin{aligned}
 \|\mathbf u\|_{X^1(J)}
 &\le C_0\|\mathbf u(t_*)\|_{H^1}
      +C_1\|\mathbf u\|_{Z'(J)}^2\\
 &\le C_0\|\mathbf u(t_*)\|_{H^1}
      +C_1\eta\|\mathbf u\|_{X^1(J)}.
\end{aligned}
\]
Choose that $\eta$ small enough, we get $\|\mathbf u\|_{X^1(J)}\lesssim \|\mathbf{u(t_{*})}\|_{H^1}$.
The number of intervals depends only on
$\|\mathbf u\|_{Z(I)}$, 
iterating from $t_0$ yields
\[
 \|\mathbf u\|_{X^1(I)}
 \le C\bigl(\|\mathbf u(t_0)\|_{H^1},
           \|\mathbf u\|_{Z(I)}\bigr).
\]
\end{proof}
\subsection{Stability}
Now, we establish stability for our system \eqref{eq:system1}.
\begin{proposition}[Stability]
\label{prop:stability-h1}
Let $I=[a,b]$ with $|I|\le1$, and let
$\widetilde{\textbf{u}}\in C(I;H^1)\cap X^1(I)$
satisfy, in the Duhamel sense,
\begin{equation}\label{error}
     (i\partial_t+\Delta)\widetilde{\textbf{u}}
 =-F(\widetilde{\textbf{u}})+\textbf{e},
\end{equation}
\begin{equation}
     \|\widetilde{\textbf{u}}\|_{Z(I)}
 +\|\widetilde{\textbf{u}}\|_{L^\infty_tH^1_x(I)}
 \le M.
\end{equation}
Suppose $t_0\in I$, $\textbf{u}_0\in H^1$, and
\[
 \|\textbf{u}_0-\widetilde{\textbf{u}}(t_0)\|_{H^1}
 +\|\textbf{e}\|_{N^1(I)}
 \le\varepsilon\le\varepsilon_1(M).
\]
Then there exists a unique strong solution $\textbf{u}$ of the system \eqref{eq:system1} on $I$, with $\textbf{u}(t_0)=\textbf{u}_0$, such that
\begin{equation}
\label{eq:stability-h1-bound}
 \|\textbf{u}\|_{X^1(I)}
 +\|\widetilde{\textbf{u}}\|_{X^1(I)}
 \le C(M),
 \qquad
 \|\textbf{u}-\widetilde{\textbf{u}}\|_{X^1(I)}
 \le C(M)\varepsilon.
\end{equation}
\end{proposition}

\begin{proof}
We first bound the approximate solution in $X^1$.
Partition $I$ on each side of $t_0$ into intervals
$I_k=[t_k,t_{k+1}]$ such that $
 \|\widetilde{\textbf{u}}\|_{Z(I_k)}\le\eta$.
By the linear estimate and
$X^1\hookrightarrow L^\infty_tH^1_x$,
\[
 \left\|
 \int_{t_k}^t e^{i(t-r)\Delta}\textbf{e}(r)\,dr
 \right\|_{X^1(I_k)}
 \lesssim\|\textbf{e}\|_{N(I)}
 \le\varepsilon.
\]
So Duhamel's formula and \Cref{lem:nonlinear-estimates} imply
\begin{align*}
 \|\widetilde{\textbf{u}}\|_{X^1(I_k)}
 \le CM
 +C\|\widetilde{\textbf{u}}\|_{Z'(I_k)}^2
 +C\varepsilon
 \le C(M+1)
 +C\eta\|\widetilde{\textbf{u}}\|_{X^1(I_k)},
\end{align*}
where we take $\varepsilon_1(M)\le1$. Choosing $C\eta\le1/2$,
we obtain
\begin{equation}\label{X1}
     \|\widetilde{\textbf{u}}\|_{X^1(I_k)}
 \le C(M+1)
\end{equation}
\begin{equation}\label{X2}
 \|\widetilde{\textbf{u}}\|_{Z'(I_k)}
 \le C\eta^{1/2}(M+1)^{1/2}.
\end{equation}

Set $\textbf{w}=\textbf{u}-\widetilde{\textbf{u}}$ and
$a_k=\|\textbf{w}(t_k)\|_{H^1}$.
On $I_k$, the error satisfies
\begin{align*}
 \textbf{w}(t)
 =e^{i(t-t_k)\Delta}\textbf{w}(t_k) +i\int_{t_k}^t e^{i(t-r)\Delta}
 \bigl[
 F(\widetilde{\textbf{u}}+\textbf{w})
 -F(\widetilde{\textbf{u}})
 \bigr](r)\,dr
 +i\int_{t_k}^t e^{i(t-r)\Delta}\textbf{e}(r)\,dr.
\end{align*}
Set
\[
\Phi(\textbf{w})  =e^{i(t-t_k)\Delta}\textbf{w}(t_k) +i\int_{t_k}^t e^{i(t-r)\Delta}
 \bigl[
 F(\widetilde{\textbf{u}}+\textbf{w})
 -F(\widetilde{\textbf{u}})
 \bigr](r)\,dr
 +i\int_{t_k}^t e^{i(t-r)\Delta}\textbf{e}(r)\,dr,
\]
where 
\begin{align*}
 F_1(\widetilde{\textbf{u}}+\textbf{w})
 -F_1(\widetilde{\textbf{u}})
 &=\widetilde u_2\overline{w_1}
   +w_2\overline{\widetilde u_1}
   +w_2\overline{w_1},\\
 F_2(\widetilde{\textbf{u}}+\textbf{w})
 -F_2(\widetilde{\textbf{u}})
 &=2\widetilde u_1w_1+w_1^2.
\end{align*}
\Cref{lem:nonlinear-estimates}, \eqref{X1}
and $X^1 \hookrightarrow Z'$  give
\[
 \|\Phi(\textbf{w})\|_{X^1(I_k)}
 \le C(a_k+\varepsilon)
 +C\|\widetilde{\textbf{u}}\|_{Z'(I_k)}
       \|\textbf{w}\|_{X^1(I_k)}
 +C\|\textbf{w}\|_{X^1(I_k)}^2.
\]
Similarly, 
\[
\begin{aligned}
 \|\Phi(\mathbf w)-\Phi(\mathbf z)\|_{X^1(I_k)}
 \le&
 C(
 \|\widetilde{\mathbf u}\|_{Z'(I_k)}
 +\|\mathbf w\|_{X^1(I_k)}
 +\|\mathbf z\|_{X^1(I_k)}
 )
 \|\mathbf w-\mathbf z\|_{X^1(I_k)}.
\end{aligned}
\]
So we can choose $a_k +\varepsilon$ and  $\eta=\eta(M)$ in \eqref{X2} small enough to make $\Phi$ is a contraction map on the closed ball
\[
 \bigl\{
 \textbf{w}\in C(I_k;H^1)\cap X^1(I_k):
 \|\textbf{w}\|_{X^1(I_k)}
 \le2C(a_k+\varepsilon)
 \bigr\}.
\]
By contraction mapping theorem, there exists the solution on $I_k$ and gives
\[
 \|\textbf{w}\|_{X^1(I_k)}
 \le2C(a_k+\varepsilon),
 \qquad
 a_{k+1}\le C(a_k+\varepsilon).
\]
Since $a_0\le\varepsilon$ and the number of intervals depends
only on $M$, choose $\varepsilon_1(M)>0$ sufficiently small such that $a_k+\varepsilon$ small enough to iterate this process.
Iteration and \eqref{X1} yields
\[
 \sum_k\|\textbf{w}\|_{X^1(I_k)}
 \le C(M)\varepsilon,
 \qquad
 \sum_k\|\widetilde{\textbf{u}}\|_{X^1(I_k)}
 \le C(M).
\]
So we prove 
\eqref{eq:stability-h1-bound}. The intervals to the left of
$t_0$ are treated inversely, and uniqueness follows
from the local theory.
\end{proof}
\section{Euclidean profile}\label{se4}
We study Euclidean profile in this section.
Fix a radial $\eta\in C_c^\infty(\mathbb R^6;[0,1])$, with
$\eta=1$ on $B(0,1)$ and $\operatorname{supp}\eta\subset B(0,2)$.
For $\boldsymbol{\phi}\in\dot H^1(\mathbb R^6;\mathbb C^2)$ and $N\ge 1$, set
\begin{equation}\label{cutoff}
 \begin{aligned}
 Q_N\boldsymbol{\phi}(x)&=\eta(x/N^{1/2})\boldsymbol{\phi}(x),\\
 \boldsymbol{\phi}_N(x)&=N^2(Q_N\boldsymbol{\phi})(Nx),\\
\mathbf{f} _N(y)&=\boldsymbol{\phi}_N(\Psi^{-1}(y)).
 \end{aligned}
\end{equation}
Here $\Psi$ satisfies $\Psi(x)=x:B(0,1) \in \mathbb{R}^6 \to\mathbb T^6$. 
\begin{remark}
    There are some basic propositions of \eqref{cutoff}.
    By Hardy's inequality, 
\begin{equation}\label{lem42-cutoff-bounds}
 \|Q_N\boldsymbol{\phi}\|_{\dot H^1}
 \lesssim\|\boldsymbol{\phi}\|_{\dot H^1},
 \qquad
 \|Q_N\boldsymbol{\phi}\|_{L^2}
 \lesssim N^{1/2}\|\boldsymbol{\phi}\|_{\dot H^1}.
\end{equation}
Moreover,
\begin{equation}\label{lem42-cutoff-convergence}
 \|Q_N\boldsymbol{\phi}-\boldsymbol{\phi}\|_{\dot H^1}\lesssim
 \|\mathbf{1}_{\{|x|\ge N^{1/2}\}}
                  \nabla\boldsymbol{\phi}\|_{L^2}+
 \|\mathbf{1}_{\{|x|\ge N^{1/2}\}}
                  |x|^{-1}\boldsymbol{\phi}\|_{L^2}
 \longrightarrow0.
\end{equation}
Changing variables in \eqref{cutoff}, we consequently obtain
\[
 \|\textbf{f}_N\|_{L^2}
 =N^{-1}\|Q_N\boldsymbol{\phi}\|_{L^2}
 \lesssim N^{-1/2}\|\boldsymbol{\phi}\|_{\dot H^1},
 \qquad
 \|\textbf{f}_N\|_{H^1}
 \lesssim\|\boldsymbol{\phi}\|_{\dot H^1}.
\]
\end{remark}
For Euclidean scattering, one can refer to \cite{GMXZ} for detailed proof.
\begin{theorem}[Euclidean scattering]
\label{thm:euclidean-scattering}
Suppose that $\boldsymbol{\psi}\in\dot H^1(\mathbb R^6;\mathbb C^2)$ satisfies
\[
 E(\boldsymbol{\psi})<E(\mathbf{W}),\qquad
 \|\nabla\psi_1\|_2^2+\frac14\|\nabla\psi_2\|_2^2<H(\mathbf{W}).
\]
Then there is a unique global scattering solution
\[
 (i\partial_t+\Delta_{\mathbb R^6})\textbf{v}
 =-F(\textbf{v}),\qquad
 \textbf{v}(0)=\boldsymbol{\psi},\qquad
 \textbf{v}\in C(\mathbb R;\dot H^1).
\]
For every Schr\"odinger-admissible pair
 $\frac2q+\frac6r=3, 2\le q\le\infty,$
one has
\begin{equation}
\label{eq:euclidean-spacetime}
 \|\nabla\textbf{v}\|_{L_t^qL_x^r}
 +\|\textbf{v}\|_{L_{t,x}^4}
 +\|F(\textbf{v})\|_{L_t^1\dot H_x^1}
 \le C_{\boldsymbol{\psi}}.
\end{equation}
There are scattering states
$\boldsymbol{\psi}^{\pm\infty}\in\dot H^1(\mathbb R^6;\mathbb C^2)$ such that
\[
 \lim_{t\to\pm\infty}
 \|\textbf{v}(t)
   -e^{it\Delta_{\mathbb R^6}}\boldsymbol{\psi}^{\pm\infty}\|_{\dot H^1}
 =0.
\]

If $\boldsymbol{\psi}\in\mathcal S(\mathbb R^6;\mathbb C^2)$, then $\textbf{v}$
is smooth and all its Schwartz seminorms are bounded on every
compact time interval.
\end{theorem}

\begin{lemma}\label{lem:euclidean-approximation}
Let $\boldsymbol{\phi}\in\dot H^1(\mathbb R^6;\mathbb C^2)$
satisfy the initial-value hypotheses of \Cref{thm:euclidean-scattering}.
Let $T_0>0$, and define $\textbf{f}_N$ from
$\boldsymbol{\phi}$ by \eqref{cutoff}. Set $I_N=[-T_0N^{-2},T_0N^{-2}]$, then
the following assertions hold.

\begin{enumerate}
\item
For all $N\ge N_0(\boldsymbol{\phi},T_0)$, there exists a unique
strong solution $\textbf{U}_N$ on $I_N$ satisfying
\[
 (i\partial_t+\Delta)\textbf{U}_N=-F(\textbf{U}_N),
 \qquad
 \textbf{U}_N(0)=\textbf{f}_N,
 \qquad
 \|\textbf{U}_N\|_{X^1(I_N)}\le C_{\boldsymbol{\phi}}.
\]

\item
Let $0<\varepsilon_1\le\varepsilon_1(\boldsymbol{\phi})$,
and choose
$\boldsymbol{\phi}'\in\mathcal S(\mathbb R^6;\mathbb C^2)$
such that
\begin{equation}\label{app}
     \|\boldsymbol{\phi}-\boldsymbol{\phi}'\|_{\dot H^1}
 \le\varepsilon_1.
\end{equation}
Let $\textbf{v}'$ be the Euclidean solution of the initial-value problem
\[(i\partial_t+\Delta_{\mathbb R^6})\textbf{v}'
 =-F(\textbf{v}'),\qquad
 \textbf{v}'(0)=\boldsymbol{\phi}'.
 \] 
 And define
\[
 \textbf{v}'_R(s,x)=\eta(x/R)\textbf{v}'(s,x),
 \qquad
 \textbf{V}_{R,N}(t,y)
 =N^2\textbf{v}'_R(N^2t,N\Psi^{-1}(y)).
\]
There exists
$R_0=R_0(T_0,\boldsymbol{\phi}',\varepsilon_1)$ such that,
for every fixed $R\ge R_0$,
\[
 \limsup_{N\to\infty}
 \|\textbf{U}_N-\textbf{V}_{R,N}\|_{X^1(I_N)}
 \le C_{\boldsymbol{\phi}}\varepsilon_1.
\]
\end{enumerate}
\end{lemma}
\begin{proof}
Write $J=[-T_0,T_0]$, so that $I_N=N^{-2}J$.
We always take $N$ sufficiently large such that
$|I_N|\le1$ and $N\ge10R$.

Choose $\varepsilon_1(\boldsymbol{\phi})\le1$ sufficiently small.
By \Cref{thm:euclidean-scattering}, there exists a solution $\mathbf{v}'$ with the initial value  $\boldsymbol{\phi}'$ satisfying
\begin{equation}\label{lem42-euclidean-bound}
 \sup_{s\in\mathbb R}\|\textbf{v}'(s)\|_{\dot H^1}
 +\|F(\textbf{v}')\|_{L_s^1\dot H_x^1(\mathbb R\times\mathbb R^6)}
 \le C_{\boldsymbol{\phi}}.
\end{equation}

Direct computation gives
\[
 (i\partial_s+\Delta_{\mathbb R^6})\textbf{v}'_R
 =-F(\textbf{v}'_R)+\textbf{e}_R,
\]
where
\[
\begin{aligned}
 (\textbf{e}_R)_1
 &=
 \bigl[\eta(x/R)^2-\eta(x/R)\bigr]
             v'_2\overline{v'_1}+
 2R^{-1}(\nabla\eta)(x/R)\cdot\nabla v'_1
 +R^{-2}(\Delta\eta)(x/R)v'_1,\\
 (\textbf{e}_R)_2
 &=
 \bigl[\eta(x/R)^2-\eta(x/R)\bigr](v'_1)^2 +
 R^{-1}(\nabla\eta)(x/R)\cdot\nabla v'_2
 +\frac12R^{-2}(\Delta\eta)(x/R)v'_2.
\end{aligned}
\]

Since $\boldsymbol{\phi}'\in\mathcal S$, all higher Sobolev
norms of $\textbf{v}'$ are bounded on $J$, thus by chain rule,
\[
\begin{aligned}
 \|\textbf{e}_R\|_{L_s^1\dot H_x^1(J)}
 &\lesssim
 \|\mathbf{1}_{\{|x|\ge R\}}
             \nabla F(\textbf{v}')\|_{L_s^1L_x^2(J)}
 +R^{-1}\|F(\textbf{v}')\|_{L_s^1L_x^2(J)}\\
 &\quad+
 R^{-1}\|\nabla^2\textbf{v}'\|_{L_s^1L_x^2(J)}
 +R^{-2}\|\nabla\textbf{v}'\|_{L_s^1L_x^2(J)}
 +R^{-3}\|\textbf{v}'\|_{L_s^1L_x^2(J)}
 \longrightarrow0
\end{aligned}
\]
as $R\to\infty$.
Also, we have 
\[
 \|F(\textbf{v}'_R)-F(\textbf{v}')\|_{L_s^1\dot H_x^1(J)} \lesssim
 \|\mathbf{1}_{\{|x|\ge R\}}
             \nabla F(\textbf{v}')\|_{L_s^1L_x^2(J)}
 +R^{-1}\|F(\textbf{v}')\|_{L_s^1L_x^2(J)}
 \longrightarrow0.
\]
And by the product rule, $\|(1-\eta(\cdot/R))\boldsymbol{\phi}'\|_{\dot H^1}\to0$.
We may therefore choose
$R_0=R_0(T_0,\boldsymbol{\phi}',\varepsilon_1)$ so that,
for every $R\ge R_0$,
\begin{equation}\label{lem42-cutoff-small}
 \|\textbf{e}_R\|_{L_s^1\dot H_x^1(J)}
 +\|F(\textbf{v}'_R)-F(\textbf{v}')\|_{L_s^1\dot H_x^1(J)}
 +\|(1-\eta(\cdot/R))\boldsymbol{\phi}'\|_{\dot H^1}
 \le\varepsilon_1.
\end{equation}

For the solution $\textbf{V}_{R,N}$, it satisfies
\[
 (i\partial_t+\Delta)\textbf{V}_{R,N}
 =-F(\textbf{V}_{R,N})+\textbf{E}_{R,N},
 \qquad
 \textbf{E}_{R,N}(t,y)
 =N^4\textbf{e}_R(N^2t,N\Psi^{-1}(y)).
\]
The scaling gives
\begin{equation}\label{lem42-error-scaling}
\begin{aligned}
 \|\textbf{E}_{R,N}\|_{L_t^1L_y^2(I_N)}
 &=N^{-1}\|\textbf{e}_R\|_{L_s^1L_x^2(J)},\\
 \|\nabla_y\textbf{E}_{R,N}\|_{L_t^1L_y^2(I_N)}
 &=\|\nabla_x\textbf{e}_R\|_{L_s^1L_x^2(J)}.
\end{aligned}
\end{equation}

We now verify the hypotheses of \Cref{prop:stability-h1}.
First, Duhamel formula, 
together with \eqref{lem42-error-scaling} give
\begin{equation}\label{lem42-X-estimate}
\begin{aligned}
 \|\textbf{V}_{R,N}\|_{X^1(I_N)} &\lesssim
\|\mathbf V_{R,N}(0)\|_{H^1}
+\|F(\mathbf V_{R,N})\|_{L_t^1H_x^1(I_N)}
+\|\mathbf E_{R,N}\|_{L_t^1H_x^1(I_N)}\\
 &\lesssim
 \|\eta(\cdot/R)\boldsymbol{\phi}'\|_{\dot H^1}
 +\|F(\textbf{v}'_R)\|_{L_s^1\dot H_x^1(J)}
 +\|\textbf{e}_R\|_{L_s^1\dot H_x^1(J)}\\
 &\quad+
 N^{-1}\Bigl(
 \|\eta(\cdot/R)\boldsymbol{\phi}'\|_{L^2}
 +\|F(\textbf{v}'_R)\|_{L_s^1L_x^2(J)}
 +\|\textbf{e}_R\|_{L_s^1L_x^2(J)}
 \Bigr).
\end{aligned}
\end{equation}

To bound the first term in \eqref{lem42-X-estimate},
note that $(\nabla\eta)(x/R)$ is supported in
$\{R\le|x|\le2R\}$, where $R^{-1}\le2|x|^{-1}$.
The product rule and Hardy's inequality therefore give
\begin{equation}\label{lem42-spatial-cutoff-bound}
\begin{aligned}
 \|\eta(\cdot/R)\boldsymbol{\phi}'\|_{\dot H^1}
 &\le
 \|\eta(\cdot/R)\nabla\boldsymbol{\phi}'\|_{L^2}
 +R^{-1}\|(\nabla\eta)(\cdot/R)\boldsymbol{\phi}'\|_{L^2}\\
 &\lesssim
 \|\nabla\boldsymbol{\phi}'\|_{L^2}
 +\||x|^{-1}\boldsymbol{\phi}'\|_{L^2}\\
 &\lesssim
 \|\boldsymbol{\phi}'\|_{\dot H^1}
 \le\|\boldsymbol{\phi}\|_{\dot H^1}+\varepsilon_1
 \le C_{\boldsymbol{\phi}},
\end{aligned}
\end{equation}
where we used
$\|\boldsymbol{\phi}-\boldsymbol{\phi}'\|_{\dot H^1}
\le\varepsilon_1\le1$.
Next, the triangle inequality,
\eqref{lem42-euclidean-bound} and
\eqref{lem42-cutoff-small} imply
\[
 \|F(\textbf{v}'_R)\|_{L_s^1\dot H_x^1(J)}
 \le C_{\boldsymbol{\phi}}+\varepsilon_1,
 \qquad
 \|\textbf{e}_R\|_{L_s^1\dot H_x^1(J)}
 \le\varepsilon_1.
\]
And it is easy to see that
\[
\|\eta(\cdot/R)\boldsymbol{\phi}'\|_{L^2}
 +\|F(\textbf{v}'_R)\|_{L_s^1L_x^2(J)}
 +\|\textbf{e}_R\|_{L_s^1L_x^2(J)}< \infty
\]
because of $\boldsymbol{\phi}' \in \mathcal{S}$.
Substituting these bounds into
\eqref{lem42-X-estimate}, and letting $N\to\infty$
with $R$ fixed, we obtain
\[
 \limsup_{N\to\infty}
 \|\textbf{V}_{R,N}\|_{X^1(I_N)}
 \le C_{\boldsymbol{\phi}}.
\]
Using $X^1\hookrightarrow L_t^\infty H_x^1$ and
$X^1 \hookrightarrow Z$, for all sufficiently large $N$,
\[
 \|\textbf{V}_{R,N}\|_{L_t^\infty H_y^1(I_N)}
 +\|\textbf{V}_{R,N}\|_{Z(I_N)}
 \le M_{\boldsymbol{\phi}}.
\]

Second, changing variables at $t=0$ gives
\begin{equation}\label{lem42-initial-comparison}
 \|\textbf{f}_N-\textbf{V}_{R,N}(0)\|_{H^1}\lesssim
 \|Q_N\boldsymbol{\phi}
             -\eta(\cdot/R)\boldsymbol{\phi}'\|_{\dot H^1} +
 N^{-1}\|Q_N\boldsymbol{\phi}
             -\eta(\cdot/R)\boldsymbol{\phi}'\|_{L^2}.
\end{equation}
For the $\dot{H}^1$ term, the triangle inequality,
\eqref{lem42-cutoff-convergence}, \eqref{app} and \eqref{lem42-cutoff-small} give
\[
 \|Q_N\boldsymbol{\phi}
             -\eta(\cdot/R)\boldsymbol{\phi}'\|_{\dot H^1}\le
 \|Q_N\boldsymbol{\phi}-\boldsymbol{\phi}\|_{\dot H^1}
 +\|\boldsymbol{\phi}-\boldsymbol{\phi}'\|_{\dot H^1}+
 \|(1-\eta(\cdot/R))\boldsymbol{\phi}'\|_{\dot H^1}
 \le o_N(1)+2\varepsilon_1.
\]
For the $L^2$ term, the triangle inequality and \eqref{lem42-cutoff-bounds} yield
\[
N^{-1}\|Q_N\boldsymbol{\phi}
             -\eta(\cdot/R)\boldsymbol{\phi}'\|_{L^2}\le
 C N^{-1/2}\|\boldsymbol{\phi}\|_{\dot H^1}
 +N^{-1}\|\eta(\cdot/R)\boldsymbol{\phi}'\|_{L^2}
 \longrightarrow0.
\]
Combine these two estimates with
\eqref{lem42-initial-comparison}, we obtain
\[
 \limsup_{N\to\infty}
 \|\textbf{f}_N-\textbf{V}_{R,N}(0)\|_{H^1}
 \lesssim\varepsilon_1.
\]

Finally, the scaling identities
\eqref{lem42-error-scaling} give :
\[
 \|\textbf{E}_{R,N}\|_{N(I_N)}\lesssim
 \|\textbf{E}_{R,N}\|_{L_t^1H_y^1(I_N)}\lesssim
 \|\textbf{e}_R\|_{L_s^1\dot H_x^1(J)}
 +N^{-1}\|\textbf{e}_R\|_{L_s^1L_x^2(J)}.
\]
By \eqref{lem42-cutoff-small},  $\|\textbf{e}_R\|_{L_s^1\dot H_x^1(J)} \le \varepsilon_1$, and $\|\textbf{e}_R\|_{L_s^1L_x^2(J)} < \infty$
since $\boldsymbol{\phi}' \in \mathcal{S}$.
Hence
\[
 \limsup_{N\to\infty}
 \|\textbf{E}_{R,N}\|_{N(I_N)}
 \lesssim\varepsilon_1.
\]

Taking $\varepsilon_1(\boldsymbol{\phi})$ sufficiently small to apply \Cref{prop:stability-h1},
we obtain, for all sufficiently large $N$, there exists the solution
with initial value $\textbf{f}_N$ and
\[
 \|\textbf{U}_N\|_{X^1(I_N)}
 \le C_{\boldsymbol{\phi}},
 \qquad
 \|\textbf{U}_N-\textbf{V}_{R,N}\|_{X^1(I_N)}
 \le C_{\boldsymbol{\phi}}\varepsilon_1.
\]
Thus we prove this Lemma.
\end{proof}
\begin{lemma}[Extinction lemma]
\label{lem:extinction}
Let $\boldsymbol{\phi}\in\dot H^1(\mathbb R^6;\mathbb C^2)$,
and define $\textbf{f}_N$ by \eqref{cutoff}.
For every $\varepsilon>0$, there exist
$T=T(\boldsymbol{\phi},\varepsilon)$ and
$N_0=N_0(\boldsymbol{\phi},\varepsilon)\ge T$ such that,
for every $N\ge N_0$,
\begin{equation}\label{eq:extinction-z}
 \|e^{it\Delta}\textbf{f}_N\|_{Z([TN^{-2},T^{-1}])}
 +\|e^{it\Delta}\textbf{f}_N\|_{Z([-T^{-1},-TN^{-2}])}
 \le\varepsilon.
\end{equation}
Moreover,
\begin{equation}\label{eq:extinction-pointwise}
 \lim_{T\to\infty}\limsup_{N\to\infty}
 \sup_{TN^{-2}\le |t|\le T^{-1}}
 \left(
 \sup_{K\ge1}K^{-2}
 \|P_K e^{it\Delta}\textbf{f}_N\|_{L^\infty}
 +\|e^{it\Delta}\textbf{f}_N\|_{L^3}
 \right)=0.
\end{equation}
\end{lemma}

\begin{proof}[Proof of \Cref{lem:extinction}]
For dyadic $M\ge1$, choose a fixed
$\chi\in C_c^\infty(\mathbb R)$, equal to one on $[-2,2]$,
and define
\[
 K_M(t,x)=\sum_{\xi\in\mathbb Z^6}
 e^{i(x\cdot\xi-t|\xi|^2)}
 \prod_{j=1}^6\chi(\xi_j/M).
\]
The one-dimensional quadratic-sum estimate in
\cite[Lemma~3.18]{Bour1993}, applied in each coordinate, gives
\begin{equation}\label{eq:extinction-rational-kernel}
 |K_M(t,x)|\lesssim
 \left[
 \frac{M}{\sqrt q\bigl(1+M|t/(2\pi)-a/q|^{1/2}\bigr)}
 \right]^6,
\end{equation}
provided that $
 (a,q)=1, 1\le q\le M,
 \left|\frac{t}{2\pi}-\frac aq\right|\le\frac1{Mq}.$
Then we use the following elementary approximation lemma.

\begin{lemma}[Dirichlet's lemma]
\label{lem:dirichlet}
For every $\alpha\in\mathbb R$ and every integer $M\ge1$,
there exist integers $a,q$ such that
\[
 (a,q)=1,\qquad 1\le q\le M,\qquad
 \left|\alpha-\frac aq\right|\le\frac1{Mq}.
\]
\end{lemma}

We claim that, for $1\le S\le M$,
\begin{equation}\label{eq:extinction-kernel}
 \sup_{SM^{-2}\le |t|\le S^{-1}}
 \|K_M(t)\|_{L^\infty}\lesssim S^{-3}M^6.
\end{equation}
Indeed, apply \Cref{lem:dirichlet} to $\alpha=t/(2\pi)$.
If $a\ne0$, then
\[
 \frac1q\le\frac{|a|}{q}
 \le\frac{|t|}{2\pi}+\frac1{Mq}
 \lesssim S^{-1}+M^{-1}\lesssim S^{-1},
\]
so $q\gtrsim S$. If $a=0$, coprimality gives $q=1$, and $
 M\left|\frac{t}{2\pi}\right|^{1/2}
 \gtrsim S^{1/2}$.

In both cases, \eqref{eq:extinction-rational-kernel} gives
\eqref{eq:extinction-kernel}. The same argument, with
$t/(4\pi)$ in place of $t/(2\pi)$, proves this estimate for
$K_M(t/2)$. Thus it applies to both component propagators.

By \eqref{lem42-cutoff-bounds}, the scaling in
\eqref{cutoff}, and \Cref{prop-pre-Stri}, it suffices by
density to consider
$\boldsymbol{\phi}\in C_c^\infty(\mathbb R^6;\mathbb C^2)$, so when $N$ is sufficiently large, $Q_N\boldsymbol{\phi}=\boldsymbol{\phi}$. Changing variables and by Plancherel, we obtain
\begin{equation}\label{eq:extinction-data-bounds}
 \|\textbf{f}_N\|_{L^1}\lesssim N^{-4},
 \qquad
 \|P_K\textbf{f}_N\|_{L^2}
 \lesssim
 N^{-1}(1+K/N)^{-10},
 \qquad
 \|\textbf{f}_N\|_{H^1}\lesssim 1.
\end{equation}
In particular, for every $p>8/3$, \Cref{prop-pre-Stri} gives
\begin{equation}\label{eq:extinction-strichartz}
 \|P_K e^{it\Delta}\textbf{f}_N\|_
 {L^p_{t,x}([-1,1]\times\mathbb T^6)}
 \lesssim
 K^{3-8/p}N^{-1}(1+K/N)^{-10}.
\end{equation}

Fix $2\le T\le N$,
we split the sum into $
 K\in[NT^{-\frac1{100}},NT^{\frac1{100}}]$ and $
 K\notin[NT^{-\frac1{100}},NT^{\frac1{100}}]$.
For $p=3$, \eqref{eq:extinction-strichartz}
and summation give
\begin{equation}\label{eq:extinction-frequency-tail}
 \sum_{K\notin[NT^{-\frac1{100}},NT^{\frac1{100}}]}
 K^{2}
 \|P_K e^{it\Delta}\textbf{f}_N\|_
 {L^3_{t,x}([TN^{-2}, T^{-1}]\times\mathbb T^6)}^3
 \lesssim T^{-\frac{3}{100}}.
\end{equation}

For $NT^{-\frac1{100}}\le K\le NT^{\frac1{100}}$, choose a dyadic
$M\sim\max(K,N)$.
Since $M\ge N\ge T$, every $t\in [TN^{-2}, T^{-1}]$ satisfies
$TM^{-2}\le |t|\le T^{-1}$.
Applying \eqref{eq:extinction-kernel} with $S=T$ and using
the uniform $L^1$ boundedness of $P_K$, we obtain
\begin{align*}
 \|P_K e^{it\Delta}\textbf{f}_N\|_
 {L^\infty_{t,x}([TN^{-2}, T^{-1}]\times\mathbb T^6)}
 \lesssim T^{-3}M^6\|P_K\textbf{f}_N\|_{L^1}
 \lesssim
 T^{-3}(K+N)^6N^{-4}.
\end{align*}
So this implies
\begin{equation}\label{eq:extinction-middle-linfty}
 K^{-2}\|P_K e^{it\Delta}\textbf{f}_N\|_
 {L^\infty_{t,x}([TN^{-2}, T^{-1}]\times\mathbb T^6)}
 \lesssim
 T^{-3}(1+K/N)^6(K/N)^{-2}
 \lesssim T^{-2}.
\end{equation}

Take $q=11/4$, so that $8/3<q<3$.
Interpolating \eqref{eq:extinction-middle-linfty} with
\Cref{prop-pre-Stri} at exponent $q$, we find for these middle frequencies,
\begin{align*}
 K^{2}\|P_K e^{it\Delta}\textbf{f}_N\|_
 {L^3_{t,x}([TN^{-2}, T^{-1}]\times\mathbb T^6)}^3
 \lesssim
 T^{-2(3-q)}
 \bigl(K\|P_K\textbf{f}_N\|_{L^2}\bigr)^q.
\end{align*}
Since $q>2$, the embedding $\ell^2\hookrightarrow\ell^q$
and \eqref{eq:extinction-data-bounds} give
\begin{equation}\label{eq:extinction-middle-sum}
\begin{aligned}
 &\quad\left(
 \sum_{NT^{-\frac1{100}}\le K\le NT^{\frac1{100}}}
 K^{2}\|P_K e^{it\Delta}\textbf{f}_N\|_
 {L^p_{t,x}([TN^{-2}, T^{-1}]\times\mathbb T^6)}^3
 \right)^{1/3}\\
 &\quad\lesssim 
 T^{-2(1-q/3)}
 \left(\sum_K
 \bigl(K\|P_K\textbf{f}_N\|_{L^2}\bigr)^q
 \right)^{1/3} \lesssim 
 T^{-2(1-q/3)}\|\textbf{f}_N\|_{H^1}^{q/3}
 \lesssim T^{-1/6}.
\end{aligned}
\end{equation}
Combining \eqref{eq:extinction-frequency-tail} and
\eqref{eq:extinction-middle-sum}, and using
\eqref{pre:Zformula}, we conclude that
\[
 \|e^{it\Delta}\textbf{f}_N\|_{Z([TN^{-2},T^{-1}])}
 \lesssim T^{-\frac1{100}}.
\]
This proves \eqref{eq:extinction-z} for smooth compactly
supported data.

We now prove \eqref{eq:extinction-pointwise}, still assuming
$\boldsymbol{\phi}\in C_c^\infty(\mathbb R^6;\mathbb C^2)$.
For $K\notin[NT^{-\frac1{100}},NT^{\frac1{100}}]$, Bernstein's inequality,
unitarity, and \eqref{eq:extinction-data-bounds} give
\[
 \sup_{t\in\mathbb R}
 K^{-2}\|P_K e^{it\Delta}\textbf{f}_N\|_{L^\infty}
 \lesssim K\|P_K\textbf{f}_N\|_{L^2}
 \lesssim
 \frac KN\left(1+\frac KN\right)^{-10}
 \lesssim T^{-\frac1{100}}.
\]
Together with \eqref{eq:extinction-middle-linfty}, this yields
\begin{equation}\label{eq:extinction-besov}
 \sup_{\substack{t\in [TN^{-2}, T^{-1}]\\K\ge1}}
 K^{-2}\|P_K e^{it\Delta}\textbf{f}_N\|_{L^\infty}
 \lesssim T^{-\frac1{100}}.
\end{equation}

For the $L^3$ estimate, we use the same frequency decomposition.
By Bernstein's inequality and
\eqref{eq:extinction-data-bounds},
\[
\begin{aligned}
 &\quad\sup_{t\in [TN^{-2}, T^{-1}]}
 \sum_{K\notin[NT^{-\frac1{100}},NT^{\frac1{100}}]}
 \|P_K e^{it\Delta}\textbf{f}_N\|_{L^3}\\&\lesssim
 \sum_{K\notin[NT^{-\frac1{100}},NT^{\frac1{100}}]}
 K\|P_K\textbf{f}_N\|_{L^2}
 \lesssim 
 \sum_{K<NT^{-\frac1{100}}}\frac KN
 +\sum_{K>NT^{\frac1{100}}}\left(\frac KN\right)^{-9}\lesssim T^{-\frac1{100}}.
\end{aligned}
\]
For $NT^{-\frac1{100}}\le K\le NT^{\frac1{100}}$, interpolation between
$L^2$ and $L^\infty$, followed by
\eqref{eq:extinction-middle-linfty} and
\eqref{eq:extinction-data-bounds}, gives
\[
\begin{aligned}
 \sup_{t\in [TN^{-2}, T^{-1}]}
 \|P_K e^{it\Delta}\textbf{f}_N\|_{L^3}
 &\le
 \left(
 \sup_{t\in [TN^{-2}, T^{-1}]}
 \|P_K e^{it\Delta}\textbf{f}_N\|_{L^\infty}
 \right)^{1/3}
 \|P_K\textbf{f}_N\|_{L^2}^{2/3}\\
 &\lesssim T^{-2/3}
 \left(\frac KN\right)^{2/3}
 \left(1+\frac KN\right)^{-20/3}.
\end{aligned}
\]
Consequently,
\[
\begin{aligned}
 \sup_{t\in [TN^{-2}, T^{-1}]}
 \sum_{NT^{-\frac1{100}}\le K\le NT^{\frac1{100}}}
 \|P_K e^{it\Delta}\textbf{f}_N\|_{L^3}
 &\lesssim 
 T^{-2/3}
 \left[
 \sum_{K\le N}\left(\frac KN\right)^{2/3}
 +\sum_{K>N}\left(\frac KN\right)^{-6}
 \right]\\
 &\lesssim T^{-2/3}.
\end{aligned}
\]
Combining the preceding estimates and using the triangle
inequality, we obtain
\[
 \sup_{t\in [TN^{-2}, T^{-1}]}
 \|e^{it\Delta}\textbf{f}_N\|_{L^3}
 \lesssim
 T^{-\frac1{100}}+T^{-2/3}
 \lesssim T^{-\frac1{100}}.
\]
Together with \eqref{eq:extinction-besov}, this proves
\eqref{eq:extinction-pointwise}.

\end{proof}
\begin{definition}
\label{def:euclidean-frames}
For $\textbf{f}\in L^2(\mathbb T^6;\mathbb C^2)$,
$t_0\in\mathbb R$, and $x_0\in\mathbb T^6$, define
\[
 (\pi_{x_0}\textbf{f})(y)=\textbf{f}(y-x_0),
 \qquad
 \Pi_{t_0,x_0}\textbf{f}
 =\pi_{x_0}e^{-it_0\Delta}\textbf{f}.
\]
For $\boldsymbol{\phi}\in\dot H^1(\mathbb R^6;\mathbb C^2)$
and $N\ge1$, let
 $T_N\boldsymbol{\phi}=\textbf{f}_N$,
where $\textbf{f}_N$ is defined by \eqref{cutoff}.
By scaling and \eqref{lem42-cutoff-bounds},
\begin{equation}\label{Zembedding}
       \|T_N\boldsymbol{\phi}\|_{H^1}
 \lesssim\|\boldsymbol{\phi}\|_{\dot H^1},
 \qquad
 \|T_N\boldsymbol{\phi}\|_{L^2}
 \lesssim N^{-1/2}\|\boldsymbol{\phi}\|_{\dot H^1}.
\end{equation}

A renormalized Euclidean frame is a sequence
$\mathcal O=(N_k,t_k,x_k)_{k\ge1}$ such that $
 N_k\longrightarrow\infty,
 t_k\longrightarrow0,
 x_k\in\mathbb T^6$,
and either $t_k=0$ for every $k$, or
$N_k^2|t_k|\to\infty$.
We denote the collection of these frames by
$\widetilde{\mathcal F}_e$.
\end{definition}

\begin{proposition}[Euclidean profiles]
\label{prop:euclidean-profiles}
Let $\mathcal O=(N_k,t_k,x_k)_{k\ge1}
\in\widetilde{\mathcal F}_e$, and let
$\boldsymbol{\phi}\in\dot H^1(\mathbb R^6;\mathbb C^2)$.

If $t_k=0$ for every $k$, assume
\[
 E(\boldsymbol{\phi})<E(\textbf{W}),
 \qquad
 H(\boldsymbol{\phi})<H(\textbf{W}).
\]
If $N_k^2|t_k|\to\infty$, assume instead
\[
 H(\boldsymbol{\phi})<E(\textbf{W}).
\]
After passing to a subsequence in the second case,
we may assume that $N_k^2t_k\to+\infty$ or
$N_k^2t_k\to-\infty$.

Let $\textbf{u}$ be the global Euclidean scattering solution
given by \Cref{thm:euclidean-scattering}, with scattering
states $\boldsymbol{\phi}^{\pm\infty}$, determined by
\[
 \begin{cases}
 \textbf{u}(0)=\boldsymbol{\phi},
     & t_k=0\ \text{for every }k,\\
 \boldsymbol{\phi}^{-\infty}=\boldsymbol{\phi},
     & N_k^2t_k\to+\infty,\\
 \boldsymbol{\phi}^{+\infty}=\boldsymbol{\phi},
     & N_k^2t_k\to-\infty.
 \end{cases}
\]
Then the following assertions hold.

\begin{enumerate}
\item
There exists $\tau=\tau(\boldsymbol{\phi})>0$ such that,
for all sufficiently large $k$, the periodic system has a
unique strong solution $\textbf{U}_k$ on $[-\tau,\tau]$
satisfying
\[
 \textbf{U}_k(0)
 =\Pi_{t_k,x_k}T_{N_k}\boldsymbol{\phi},
\]
and
\begin{equation}\label{eq:euclidean-profile-bound}
 \|\textbf{U}_k\|_{X^1([-\tau,\tau])}
 \le C_{\boldsymbol{\phi}}.
\end{equation}

\item
For every $\varepsilon>0$, one can choose a smooth global
Euclidean solution $\textbf{u}'$, with Schwartz initial
data and scattering states
$\boldsymbol{\phi}'^{\,\pm\infty}$, such that
\[
 \sup_{s\in\mathbb R}
 \|\textbf{u}'(s)-\textbf{u}(s)\|_{\dot H^1}
 +\sum_{\pm}
 \|\boldsymbol{\phi}'^{\,\pm\infty}
       -\boldsymbol{\phi}^{\pm\infty}\|_{\dot H^1}
 \le\varepsilon,
\]
and a number
$T_0=T_0(\boldsymbol{\phi},\varepsilon)\ge \tau^{-1}$
with the following property.

For every fixed $T\ge T_0$, there exists
$R_0=R_0(\boldsymbol{\phi},\varepsilon,T)$ such that,
for every fixed $R\ge R_0$ and all sufficiently large $k$,
\begin{equation}\label{eq:euclidean-profile-core}
 \|\textbf{U}_k-\widetilde{\textbf{u}}_k\|_
 {X^1(\{|t-t_k|\le TN_k^{-2}\}
             \cap\{|t|\le T^{-1}\})}
 \le\varepsilon,
\end{equation}
where
\[
 (\pi_{-x_k}\widetilde{\textbf{u}}_k)(t,y)
 =
 N_k^2\eta\!\left(\frac{N_k\Psi^{-1}(y)}R\right)
 \textbf{u}'\!
 \left(N_k^2(t-t_k),N_k\Psi^{-1}(y)\right).
\]
Moreover, 
\begin{equation}\label{eq:euclidean-profile-tail}
 \|\textbf{U}_k
   -\Pi_{t_k-t,x_k}T_{N_k}\boldsymbol{\phi}^{\pm\infty}\|_
 {X^1(\{\pm(t-t_k)\ge TN_k^{-2}\}
             \cap\{|t|\le T^{-1}\})}
 \le\varepsilon.
\end{equation}
\end{enumerate}
\end{proposition}
\begin{proof}
By translation, assume that $x_k=0$.

\noindent
First suppose that $t_k=0$.
By \Cref{thm:euclidean-scattering} and
\Cref{lem:extinction}, choose $\textbf{u}'$ and then $T_0$
sufficiently large so that
\[
\begin{aligned}
 &\sup_s\|\textbf{u}'(s)-\textbf{u}(s)\|_{\dot H^1}
 +\sum_\pm
 \|\boldsymbol{\phi}'^{\,\pm\infty}
       -\boldsymbol{\phi}^{\pm\infty}\|_{\dot H^1}
 \le\delta,\\
 &\sum_\pm
 \|\textbf{u}(\pm T_0)
       -e^{\pm iT\Delta_{\mathbb R^6}}
          \boldsymbol{\phi}^{\pm\infty}\|_{\dot H^1}
 \le\delta,\\
 &\limsup_{k\to\infty}\sum_\pm
 \|e^{it\Delta}T_{N_k}\boldsymbol{\phi}^{\pm\infty}\|_
 {Z(\{T_0N_k^{-2}\le\pm t\le T_0^{-1}\})}
 \le\delta.
\end{aligned}
\]
Here $\textbf{u}'(0)$ is Schwartz and $\textbf{u}(0)=\boldsymbol{\phi}$.

For fixed $T_0$, \Cref{lem:euclidean-approximation} gives,
for $R_0$ and then $k$ sufficiently large,
\begin{equation}
       \|\textbf{U}_k\|_{X^1([-T_0 N_k^{-2},T_0 N_k^{-2}])}
 \le C_{\boldsymbol{\phi}},\label{gl}
\end{equation}
\begin{equation}
       \|\textbf{U}_k-\widetilde{\textbf{u}}_k\|_
 {X^1([-T_0 N_k^{-2},T_0 N_k^{-2}])}
 \le C_{\boldsymbol{\phi}}\delta.\label{ap}
\end{equation}
By \eqref{Zembedding} and increasing $R_0$ large enough such that
\[
\begin{aligned}
 &\quad \|\textbf{U}_k(\pm T_0 N_k^{-2})
       -e^{\pm iT_0 N_k^{-2}\Delta}
          T_{N_k}\boldsymbol{\phi}^{\pm\infty}\|_{H^1}\\
 &\lesssim
 \|\textbf{U}_k-\widetilde{\textbf{u}}_k\|_
      {X^1([-T_0 N_k^{-2},T_0 N_k^{-2}])}
 +\|(1-\eta(\cdot/R))\textbf{u}'(\pm T_0)\|_{\dot H^1}\\
 &\quad +\|\textbf{u}'(\pm T_0)-\textbf{u}(\pm T_0)\|_{\dot H^1}
 +\|\textbf{u}(\pm T_0)
       -e^{\pm iT_0\Delta_{\mathbb R^6}}
          \boldsymbol{\phi}^{\pm\infty}\|_{\dot H^1}
 +o_k(1)\\
 &\le C_{\boldsymbol{\phi}}\delta+o_k(1).
\end{aligned}
\]

By \Cref{prop-pre-Xs-Hs},
\[
\begin{aligned}
 &\quad\left\|
 e^{i(t\mp T_0 N_k^{-2})\Delta}
       \textbf{U}_k(\pm T_0 N_k^{-2})
 -
 e^{it\Delta}T_{N_k}\boldsymbol{\phi}^{\pm\infty}
 \right\|_{X^1([-T_0^{-1},T_0^{-1}])}\\
 &\le
 \|\textbf{U}_k(\pm TN_k^{-2})
       -e^{\pm iTN_k^{-2}\Delta}
          T_{N_k}\boldsymbol{\phi}^{\pm\infty}\|_{H^1}
 \le C_{\boldsymbol{\phi}}\delta+o_k(1).
\end{aligned}
\]
Consequently, \eqref{eq:z-prime-norm} and extinction lemma give
\[
 \left\|
 e^{i(t\mp T_0 N_k^{-2})\Delta}
       \textbf{U}_k(\pm T_0 N_k^{-2})
 \right\|_{Z'(\{T_0 N_k^{-2}\le\pm t\le T_0^{-1}\})}
 \le C_{\boldsymbol{\phi}}\delta^{1/2}+o_k(1).
\]
Thus \Cref{prop:local-theory-h1} and
\eqref{eq:local-h1-correction} imply
\begin{equation}\
       \|\textbf{U}_k\|_{X^1([-T_0^{-1},T_0^{-1}])}
 \le C_{\boldsymbol{\phi}},\label{gl2}
\end{equation}
\begin{equation}
       \|\textbf{U}_k-e^{it\Delta}
        T_{N_k}\boldsymbol{\phi}^{\pm\infty}\|_
 {X^1(\{T_0 N_k^{-2}\le\pm t\le T_0 ^{-1}\})}
 \lesssim_{\boldsymbol{\phi}}
 (\delta^{1/2})^2+\delta+o_k(1).\label{ap2}
\end{equation}
For arbitrary $\varepsilon>0$, choose $
 C_{\boldsymbol{\phi}}\delta<\varepsilon/2$,
 $T\ge T_0\ge\tau^{-1}$, 
 $R\ge R_0(\boldsymbol{\phi},\delta,T)$
and then $k$ sufficiently large.
So \eqref{gl2}, \eqref{ap} and \eqref{ap2}  give \eqref{eq:euclidean-profile-bound},
\eqref{eq:euclidean-profile-core} and
\eqref{eq:euclidean-profile-tail}.

\smallskip
\noindent
Now suppose that $N_k^2t_k\to+\infty$.
The Euclidean wave operator gives
\[
 \|\textbf{u}(s)
       -e^{is\Delta_{\mathbb R^6}}\boldsymbol{\phi}\|_{\dot H^1}
 \longrightarrow0
 \qquad(s\to-\infty),
\]
and hence
\[
 E(\textbf{u})=H(\boldsymbol{\phi})<E(\textbf{W}),
 \qquad
 H(\textbf{u}(s))\longrightarrow H(\boldsymbol{\phi})
       <H(\textbf{W}).
\]
By \Cref{thm:euclidean-scattering} and energy trapping,
$\textbf{u}$ extends globally and satisfies the threshold
conditions. The first case, applied to $\textbf{u}(0)$, gives $
 \textbf{V}_k(0)=T_{N_k}\textbf{u}(0)$ and 
 $\|\textbf{V}_k\|_{X^1([-2\tau,2\tau])}
 \le C_{\boldsymbol{\phi}}.$

For every fixed $T$ and all sufficiently large $k$, $TN_k^{-2}\le t_k\le T^{-1}$, so \eqref{ap2} yields
\[
\limsup_{k\to\infty}
 \|\textbf{V}_k(-t_k)
       -e^{-it_k\Delta}T_{N_k}\boldsymbol{\phi}\|_{H^1} \lesssim
 \limsup_{k\to\infty}
 \|\textbf{V}_k-e^{it\Delta}T_{N_k}\boldsymbol{\phi}\|_
 {X^1([-T^{-1},-TN_k^{-2}])}
 \le\varepsilon.
\]
Since $\varepsilon$ is arbitrary,
\Cref{prop:stability-h1} gives
\[
 \|\textbf{U}_k-\textbf{V}_k(\,\cdot-t_k)\|_
 {X^1([-\tau,\tau])}\longrightarrow0,
 \qquad
 \|\textbf{U}_k\|_{X^1([-\tau,\tau])}
 \le C_{\boldsymbol{\phi}}.
\]
We notice that $
 \{|t|\le T^{-1},\ \pm(t-t_k)\ge TN_k^{-2}\}
 \subset
 \{|t-t_k|\le2T^{-1},\
       \pm(t-t_k)\ge(T/2)N_k^{-2}\}.$
So the first case gives \eqref{eq:euclidean-profile-core} and
\eqref{eq:euclidean-profile-tail}.
The case $N_k^2t_k\to-\infty$ follows by using the
forward wave operator. Translating back completes the proof.
\end{proof}

\begin{corollary}[Decomposition of the nonlinear Euclidean profiles]
\label{euclidean}
Let $\textbf{U}_k$ be as in \Cref{prop:euclidean-profiles}.
For every $\theta>0$, there exist $T_\theta,R_\theta\ge1$
such that, for all sufficiently large $k$, on
$I_\theta=(-T_\theta^{-1},T_\theta^{-1})$,
\begin{equation}\label{omega-decomposition}
 \textbf{U}_k
 =\boldsymbol{\omega}_k^{\theta,-\infty}
  +\boldsymbol{\omega}_k^{\theta,+\infty}
  +\boldsymbol{\omega}_k^\theta
  +\boldsymbol{\rho}_k^\theta,
\end{equation}
where
\begin{equation}\label{omega-bounds}
\begin{aligned}
 &\sum_\pm
 \|\boldsymbol{\omega}_k^{\theta,\pm\infty}\|_{Z'(I_\theta)}
 +\|\boldsymbol{\rho}_k^\theta\|_{X^1(I_\theta)}
 \le\theta,\\
 &\sum_\pm
 \|\boldsymbol{\omega}_k^{\theta,\pm\infty}\|_{X^1(I_\theta)}
 +\|\boldsymbol{\omega}_k^\theta\|_{X^1(I_\theta)}
 \le C_{\boldsymbol{\phi}},\\
 &\boldsymbol{\omega}_k^{\theta,\pm\infty}
 =P_{\le R_\theta N_k}
       \boldsymbol{\omega}_k^{\theta,\pm\infty},\\
 &|\nabla_x^m\boldsymbol{\omega}_k^\theta|
 +N_k^{-2}|\partial_t\nabla_x^m
                    \boldsymbol{\omega}_k^\theta|
 \le R_\theta N_k^{|m|+2}\mathbf{1}_{S_k^\theta},
 \qquad 0\le |m|\le10,
\end{aligned}
\end{equation}
and
\[
 S_k^\theta
 =\left\{(t,x)\in I_\theta\times\mathbb T^6:
 |t-t_k|\le T_\theta N_k^{-2},\
 |x-x_k|\le R_\theta N_k^{-1}\right\}.
\]
\end{corollary}
\begin{proof}
By \Cref{prop:euclidean-profiles}, for any
$\varepsilon>0$, all sufficiently large $T$, and then all
sufficiently large $R$, we have
\[
 \|\textbf{U}_k-\widetilde{\textbf{u}}_k\|_
 {X^1(\{|t-t_k|\le TN_k^{-2}\}\cap\{|t|\le T^{-1}\})}
 \le\varepsilon
\]
for all large $k$, where
\[
 (\pi_{-x_k}\widetilde{\textbf{u}}_k)(t,y)
 =N_k^2\eta\!\left(\frac{N_k\Psi^{-1}(y)}R\right)
 \textbf{u}'\!\left(N_k^2(t-t_k),N_k\Psi^{-1}(y)\right),
\]
and $\textbf{u}'$ is the Euclidean approximation in \Cref{prop:euclidean-profiles}. In addition,
\[
 \|\textbf{U}_k-
   \Pi_{t_k-t,x_k}T_{N_k}\boldsymbol{\phi}^{\pm\infty}\|_
 {X^1(\{\pm(t-t_k)\ge TN_k^{-2}\}\cap\{|t|\le T^{-1}\})}
 \le\varepsilon.
\]

Choose $T_\theta$ sufficiently large, using
\Cref{lem:extinction}, such that
\[
 \sum_\pm
 \|\Pi_{t_k-t,x_k}T_{N_k}\boldsymbol{\phi}^{\pm\infty}\|_
 {Z(\{\pm(t-t_k)\ge T_\theta N_k^{-2}\}\cap I_\theta)}
 \le\varepsilon^2
\]
for all large $k$.
Then choose $R$ for \Cref{prop:euclidean-profiles} with
$T=T_\theta$. Define the three parts as follows.

\begin{enumerate}
\item
Set $
 \boldsymbol{\omega}_k^{\theta,\pm\infty}
 :=\mathbf{1}_{\{\pm(t-t_k)\ge T_\theta N_k^{-2}\}}(t)
   \Pi_{t_k-t,x_k}P_{\le LN_k}T_{N_k}
                         \boldsymbol{\phi}^{\pm\infty}
 \quad\text{on }I_\theta.$
Choose $R_\theta\ge2R$ large enough that
$P_{\le R_\theta N_k}P_{\le LN_k}=P_{\le LN_k}$.
It follows that $
 \boldsymbol{\omega}_k^{\theta,\pm\infty}
 =P_{\le R_\theta N_k}\boldsymbol{\omega}_k^{\theta,\pm\infty}.$
\item
Set $
 \boldsymbol{\omega}_k^\theta
 :=\widetilde{\textbf{u}}_k\mathbf{1}_{S_k^\theta}.$
By \Cref{lem:euclidean-approximation},
\[
\begin{aligned}
 |\nabla_x^m\boldsymbol{\omega}_k^\theta|
  +N_k^{-2}\mathbf{1}_{S_k^\theta}
       |\partial_t\nabla_x^m\boldsymbol{\omega}_k^{\theta}| \le R_\theta N_k^{|m|+2}\mathbf{1}_{S_k^\theta},
 \qquad 0\le|m|\le10.
\end{aligned}
\]

\item
Set $
 \boldsymbol{\rho}_k^\theta
 :=\textbf{U}_k-\boldsymbol{\omega}_k^{\theta,-\infty}
              -\boldsymbol{\omega}_k^{\theta,+\infty}
              -\boldsymbol{\omega}_k^\theta.$ 
By \Cref{prop:euclidean-profiles}, 
 $\|\boldsymbol{\rho}_k^\theta\|_{X^1(I_\theta)}
 \le C\varepsilon.$
\end{enumerate}

Then we have 
\[
 \sum_\pm
 \|\boldsymbol{\omega}_k^{\theta,\pm\infty}\|_{X^1(I_\theta)}
 +\|\boldsymbol{\omega}_k^\theta\|_{X^1(I_\theta)}
 \le C_{\boldsymbol{\phi}}.
\]
Due to \Cref{lem:extinction},
\[
 \sum_\pm
 \|\boldsymbol{\omega}_k^{\theta,\pm\infty}\|_{Z(I_\theta)}
 \le C\varepsilon^2.
\]
Consequently, \eqref{eq:z-prime-norm} implies
\[
 \sum_\pm
 \|\boldsymbol{\omega}_k^{\theta,\pm\infty}\|_{Z'(I_\theta)}
 +\|\boldsymbol{\rho}_k^\theta\|_{X^1(I_\theta)}
 \le C_{\boldsymbol{\phi}}\varepsilon\le\theta,
\]
by choosing $\varepsilon(\theta,\boldsymbol{\phi})$ sufficiently small
. This completes the proof.
\end{proof} 

\section{Profile decomposition}\label{section5}
In this section, we construct the profile decomposition on $\mathbb{T}^6$.

We recall the operators $T_N$ and
$\Pi_{t,x}=\pi_xe^{-it\Delta}$ from
\Cref{def:euclidean-frames}.
\begin{definition}[Euclidean frames]\label{euclidean-frames}
(1) A Euclidean frame is a sequence
$\mathcal O=(N_k,t_k,x_k)_k$ such that $
 N_k\longrightarrow\infty,
 t_k\longrightarrow0,
 x_k\in\mathbb T^6.$
Two frames
$\mathcal O=(N_k,t_k,x_k)_k$ and
$\mathcal O'=(M_k,s_k,y_k)_k$ are orthogonal if
\begin{equation}\label{frame-distance}
 \left|\log\frac{N_k}{M_k}\right|
 +N_k^2|t_k-s_k|
 +N_kd_{\mathbb T^6}(x_k,y_k)
 \longrightarrow\infty.
\end{equation}
They are equivalent if the expression on the left-hand
side of \eqref{frame-distance} remains bounded.
After passing to a subsequence, any two frames are
either equivalent or orthogonal.

(2) For a frame $\mathcal O=(N_k,t_k,x_k)_k$ and
$\boldsymbol{\phi}\in\dot H^1(\mathbb R^6;\mathbb C^2)$,
the corresponding Euclidean profile is
\[
 \widetilde{\boldsymbol{\phi}}_{\mathcal O_k}
 :=
 \Pi_{t_k,x_k}T_{N_k}\boldsymbol{\phi}.
\]
\end{definition}

\begin{proposition}[Equivalence of frames]\label{equivalence}
(1) If $\mathcal O$ and $\mathcal O'$ are equivalent, then,
after passing to a further subsequence, there exists
a surjective isometry
\[
 T:\dot H^1(\mathbb R^6;\mathbb C^2)
 \longrightarrow
 \dot H^1(\mathbb R^6;\mathbb C^2),
\]
preserving each component gradient norm, such that,
for every $\boldsymbol{\phi}\in\dot H^1$,
\[
 \left\|
 \widetilde{T\boldsymbol{\phi}}_{\mathcal O_k}
 -
 \widetilde{\boldsymbol{\phi}}_{\mathcal O'_k}
 \right\|_{H^1}
 \longrightarrow0.
\]

(2)If $\mathcal O$ and $\mathcal O'$ are orthogonal, then,
for every
$\boldsymbol{\phi},\boldsymbol{\psi}
 \in\dot H^1(\mathbb R^6;\mathbb C^2)$
and $j,\ell\in\{1,2\}$,
\begin{align}
 \left\langle
 (\widetilde{\boldsymbol{\phi}}_{\mathcal O_k})_j,
 (\widetilde{\boldsymbol{\psi}}_{\mathcal O'_k})_j
 \right\rangle_{H^1}
 &\longrightarrow0,
 \label{linear-orth}\\
 \int_{\mathbb T^6}
 |(\widetilde{\boldsymbol{\phi}}_{\mathcal O_k})_j|^{3/2}
 |(\widetilde{\boldsymbol{\psi}}_{\mathcal O'_k})_\ell|^{3/2}
 \,dx
 &\longrightarrow0.
 \label{linear-product}
\end{align}
\end{proposition}
\begin{proof}
	See Lemma 5.4 in \cite{IPS2012} for the proof.
\end{proof}

\begin{definition}[Absence from a frame]\label{absence}
A bounded sequence
$\{\textbf{f}_k\}_k\subset H^1(\mathbb T^6;\mathbb C^2)$
is absent from a Euclidean frame $\mathcal O$ if,
for every
$\boldsymbol{\psi}\in\dot H^1(\mathbb R^6;\mathbb C^2)$
and $j=1,2$,
\[
\begin{aligned}
 \left\langle
 (\textbf{f}_k)_j,
 (\widetilde{\boldsymbol{\psi}}_{\mathcal O_k})_j
 \right\rangle_{L^2}
 &\longrightarrow0,\\
 \left\langle
 (\textbf{f}_k)_j,
 (\widetilde{\boldsymbol{\psi}}_{\mathcal O_k})_j
 \right\rangle_{H^1}
 &\longrightarrow0.
\end{aligned}
\]
\end{definition}

\begin{remark}
	The $L^2$ condition holds automatically for bounded
$H^1$ sequences. Indeed,
\[
\begin{aligned}
 \left|
 \left\langle
 (\textbf{f}_k)_j,
 (\widetilde{\boldsymbol{\psi}}_{\mathcal O_k})_j
 \right\rangle_{L^2}
 \right| \le
 \|\textbf{f}_k\|_{L^2}
 \|T_{N_k}\boldsymbol{\psi}\|_{L^2} 
 \longrightarrow0.
\end{aligned}
\]
Thus absence is equivalently expressed by the
vanishing of all corresponding gradient pairings.

\end{remark}
\begin{lemma}[Refined Strichartz inequality]\label{refined}
For $\textbf{f}\in H^1(\mathbb T^6;\mathbb C^2)$
and an interval $I$ with $|I|\le1$, one has
\begin{equation}\label{refined-display}
 \|e^{it\Delta}\textbf{f}\|_{Z(I)}
 \lesssim
 \|\textbf{f}\|_{H^1}^{11/(12)}
 \left(
 \max_{j=1,2}\sup_N
 N^{-2}
 \|P_N(e^{it\Delta}\textbf{f})_j
       \|_{L^\infty(I\times\mathbb T^6)}
 \right)^{1/12}.
\end{equation}
\end{lemma}

\begin{proof}
We use \Cref{prop-pre-Stri} at exponent
$11/4>8/3$. Its frequency-supported form gives
\[
 \|P_N(e^{it\Delta}\textbf{f})_j
       \|_{L^{11/4}(I\times\mathbb T^6)}
 \lesssim
 N^{3-32/11}\|P_Nf_j\|_{L^2}.
\]

For the second component of \eqref{eq:system1}, use the time change
$t\mapsto t/2$.
Time translation and unitarity allow any interval
of length at most one.

By H\"older's inequality, \Cref{prop-pre-Stri} gives
\[
\begin{aligned}
 &\quad \sum_N N^{2}
 \|P_N(e^{it\Delta}\textbf{f})_j
       \|_{L^3(I\times\mathbb T^6)}^3\\
 &\le
 ( \max_{j=1,2}\sup_N
 N^{-2}
 \|P_N(e^{it\Delta}\textbf{f})_j
       \|_{L^\infty(I\times\mathbb T^6)})^{1/4}
 \sum_N N^{8-11/2}
 \|P_N(e^{it\Delta}\textbf{f})_j
       \|_{L^{11/4}(I\times\mathbb T^6)}^{11/4}\\
 &\lesssim
 ( \max_{j=1,2}\sup_N
 N^{-2}
 \|P_N(e^{it\Delta}\textbf{f})_j
       \|_{L^\infty(I\times\mathbb T^6)})^{1/4}
 \sum_N
 \bigl(N\|P_Nf_j\|_{L^2}\bigr)^{11/4}\\
 &\lesssim
 ( \max_{j=1,2}\sup_N
 N^{-2}
 \|P_N(e^{it\Delta}\textbf{f})_j
       \|_{L^\infty(I\times\mathbb T^6)})^{1/4}
 \left(
 \sum_N N^2\|P_Nf_j\|_{L^2}^2
 \right)^{11/8}\\
 &\lesssim
 ( \max_{j=1,2}\sup_N
 N^{-2}
 \|P_N(e^{it\Delta}\textbf{f})_j
       \|_{L^\infty(I\times\mathbb T^6)})^{1/4}\|f_j\|_{H^1}^{11/4}.
\end{aligned}
\]
Combing with
\eqref{pre:Zformula} proves the result.
\end{proof}

\begin{proposition}[Profile decomposition]\label{decomp}
Let $\{\textbf{f}_k\}\subset H^1(\mathbb T^6;\mathbb C^2)$ satisfy
\begin{equation}\label{bound}
       \limsup_{k\to\infty}\|\textbf{f}_k\|_{H^1}\le A<\infty,
\end{equation}
and let $I_k=(-T_k,T^k)\ni0$ satisfy $|I_k|\to0$.
After passing to a subsequence,
$\textbf{f}_k\rightharpoonup\textbf{g}$ in $H^1$, and there
exist $J^*\in\mathbb N\cup\{0,\infty\}$, pairwise orthogonal
Euclidean frames
$\mathcal O^\alpha=(N_k^\alpha,t_k^\alpha,x_k^\alpha)_k$,
and
$\boldsymbol{\psi}^\alpha\in\dot H^1(\mathbb R^6;\mathbb C^2)$
such that, for every finite $0\le J\le J^*$,
\begin{equation}\label{profile-sum}
 \textbf{f}_k
 =\textbf{g}
  +\sum_{1\le\alpha\le J}
       \widetilde{\boldsymbol{\psi}}^\alpha_{\mathcal O_k^\alpha}
  +\textbf{R}_k^J.
\end{equation}
The frames may be chosen renormalized: for each $\alpha$,
either
\[
 t_k^\alpha=0\quad\text{for every }k,
 \qquad\text{or}\qquad
 (N_k^\alpha)^2t_k^\alpha\longrightarrow+\infty
 \quad\text{or}\quad-\infty.
\]
For each fixed $J$, the remainder $\textbf{R}_k^J$
converges weakly to zero in $H^1$ and is absent from
$\mathcal O^1,\ldots,\mathcal O^J$.
Moreover,
\begin{align}
 \lim_{J\to J^*}\limsup_{k\to\infty}
   \|e^{it\Delta}\textbf{R}_k^J\|_{Z(I_k)}=0,
 \label{remainder-small}
\end{align}
For $j=1,2$ and every fixed $J$,
\begin{align}
 \|(\textbf{f}_k)_j\|_2^2
 &=\|g_j\|_2^2+\|(\textbf{R}_k^J)_j\|_2^2+o_k(1),
 \label{mass-pythagoras}\\
 \|\nabla(\textbf{f}_k)_j\|_2^2
 &=\|\nabla g_j\|_2^2
   +\sum_{\alpha\le J}
       \|\nabla\psi_j^\alpha\|_{L^2(\mathbb R^6)}^2
   +\|\nabla(\textbf{R}_k^J)_j\|_2^2+o_k(1).
 \label{gradient-pythagoras}
\end{align}
The potential energy term satisfies
\begin{align}
 \operatorname{Re}\int_{\mathbb T^6}
       \overline{(\textbf{f}_k)_2}(\textbf{f}_k)_1^2\,dx
 &=\operatorname{Re}\int_{\mathbb T^6}
       \overline{g_2}g_1^2\,dx
   +\sum_{\substack{\alpha\le J\\t_k^\alpha=0}}
       \operatorname{Re}\int_{\mathbb R^6}
       \overline{\psi_2^\alpha}(\psi_1^\alpha)^2\,dx
 \notag\\
 &\quad+
   \operatorname{Re}\int_{\mathbb T^6}
       \overline{(\textbf{R}_k^J)_2}
                  (\textbf{R}_k^J)_1^2\,dx
   +o_k(1).
 \label{cubic-decoupling}
\end{align}
Consequently, the modified energy satisfies
\begin{align}
 E(\textbf{f}_k)+c_0M(\textbf{f}_k)
 &=E(\textbf{g})+c_0M(\textbf{g})
   +\sum_{\substack{\alpha\le J\\t_k^\alpha=0}}
       E_{\mathbb R^6}(\boldsymbol{\psi}^\alpha)
 \notag\\
 &\quad+
   \sum_{\substack{\alpha\le J\\
          |(N_k^\alpha)^2t_k^\alpha|\to\infty}}
   \left(
     \|\nabla\psi_1^\alpha\|_{L^2(\mathbb R^6)}^2
     +\frac14\|\nabla\psi_2^\alpha\|_{L^2(\mathbb R^6)}^2
   \right)
 \notag\\
 &\quad+
   E(\textbf{R}_k^J)+c_0M(\textbf{R}_k^J)+o_k(1),
 \label{energy-decoupling}
\end{align}
where $c_0$ is the fixed constant in \eqref{eq:torus-sharp-sobolev}.
\end{proposition}

\begin{lemma}\label{extraction}
Under the hypotheses of \Cref{decomp}, for every
$0<\delta\le1$, after passing to a subsequence, there are
$0\le J\lesssim \delta^{-2}$ pairwise orthogonal Euclidean frames
$\mathcal O^1,\ldots,\mathcal O^J$ and profiles
$ \widetilde{\boldsymbol\psi}^{\alpha}_{\mathcal O_k^\alpha}\in\dot H^1(\mathbb R^6;\mathbb C^2)$ such that
\[
 \textbf{f}_k=\textbf{g}+
 \sum_{1\le\alpha\le J}
 \widetilde{\boldsymbol\psi}^{\alpha}_{\mathcal O_k^\alpha}
 +\textbf{R}_k,
 \qquad
 \limsup_{k\to\infty}
 \|e^{it\Delta}\textbf{R}_k\|_{Z(I_k)}\lesssim \delta^{O(1)}.
\]
Moreover, for $j=1,2$,
\[
\begin{aligned}
 \|(\textbf{f}_k)_j\|_2^2
 &=\|g_j\|_2^2+\|(\textbf{R}_k)_j\|_2^2+o_k(1),\\
 \|\nabla(\textbf{f}_k)_j\|_2^2
 &=\|\nabla g_j\|_2^2+
 \sum_{\alpha=1}^J\|\nabla\psi_j^\alpha\|_{L^2(\mathbb R^6)}^2
 +\|\nabla(\textbf{R}_k)_j\|_2^2+o_k(1).
\end{aligned}
\]
\end{lemma}

\begin{proof}
We assume $A>0$.
For a bounded sequence $\{\textbf{h}_k\}\subset H^1$, set
\[
 \Omega(\{\textbf{h}_k\})
 :=\limsup_{k\to\infty}\max_{j=1,2}
 \sup_{\substack{N\ge1,\ t\in I_k\\x\in\mathbb T^6}}
 N^{-2}|P_N(e^{it\Delta}\textbf{h}_k)_j(x)|.
\]
Because $H^1(\mathbb{T}^6)$ is a separable, up to a
subsequence, there exists $\textbf{g}\in H^1$ such that
$\textbf{f}_k\rightharpoonup\textbf{g}$ in $H^1$.
Writing $\widetilde{\textbf{f}}_k=\textbf{f}_k-\textbf{g}$, we have
$\widetilde{\textbf{f}}_k\rightharpoonup0$ and, for $j=1,2$,
\[
\begin{aligned}
 \|(\textbf{f}_k)_j\|_2^2
 &=\|g_j\|_2^2+\|(\widetilde{\textbf{f}}_k)_j\|_2^2+o_k(1),\\
 \|\nabla(\textbf{f}_k)_j\|_2^2
 &=\|\nabla g_j\|_2^2+
   \|\nabla(\widetilde{\textbf{f}}_k)_j\|_2^2+o_k(1).
\end{aligned}
\]
If $\Omega(\{\widetilde{\textbf{f}}_k\})\le\delta$, no
Euclidean profile needs to be extracted. Otherwise, after passing
to a subsequence, there are a fixed $j\in\{1,2\}$ and
$(N_k,t_k,x_k)$, with $t_k\in I_k$, such that
\begin{equation}\label{selected-point}
 N_k^{-2}|P_{N_k}(e^{it_k\Delta}\widetilde{\textbf{f}}_k)_j(x_k)|
 \ge\frac12\Omega(\{\widetilde{\textbf{f}}_k\})
 >\frac12\delta.
\end{equation}
Now we claim that $N_k\to\infty$. In fact, if $N_k$ remains
bounded, then, up to a subsequence, we may assume that $
 N_k\to N_\infty<\infty$ and 
 $x_k\longrightarrow x_0\in\mathbb T^6.$
In this case, define
$
 \psi'_j=(1-\Delta_{\mathbb T^6})^{-1}
          N_\infty^{-2}P_{N_\infty}\delta_{x_0}$.
So 
\eqref{selected-point} implies, for all sufficiently large $k$,
\[
 \left|
 \left\langle
 \widetilde{\textbf{f}}_k,\boldsymbol{\psi'}
 \right\rangle_{H^1}'
 \right|
 \ge\frac{\delta}{4},
\]
which contradicts
$\widetilde{\textbf{f}}_k\rightharpoonup0$ in $H^1(\mathbb T^6)$.
We then define the Euclidean frame
$\mathcal O=(N_k,t_k,x_k)_k$ and the function
\[
 \psi_j=\mathcal F_{\mathbb R^6}^{-1}
 \left(
 \frac{\eta(\xi)-\eta(2\xi)}{|\xi|^2}
 \right),
\]
By Parseval's identity, the definition of $T_N$, and the
rapid decay of $\psi_j$, we have
\[
 \lim_{k\to\infty}
 \left\|
 (1-\Delta_{\mathbb T^6})(T_{N_k}\boldsymbol{\psi})_j
 -N_k^{-2}P_{N_k}\delta_{0}
 \right\|_{L^{3/2}(\mathbb T^6)}
 =0.
\]
Hence, by Sobolev embedding,
\[
 \lim_{k\to\infty}
 \left\|
 (1-\Delta_{\mathbb T^6})(T_{N_k}\boldsymbol{\psi})_j
 -N_k^{-2}P_{N_k}\delta_{0}
 \right\|_{H^{-1}(\mathbb T^6)}
 =0.
\]
Consequently,
\begin{equation}\label{fixed-profile-test}
 \frac12\Omega(\{\widetilde{\textbf{f}}_k\})
 \le
 \left|
 \left\langle
 \widetilde{\textbf{f}}_k,
 \widetilde{\boldsymbol{\psi}}_{\mathcal O_k}
 \right\rangle_{H^1}
 \right|
 +o_k(1).
\end{equation}
Also, by \eqref{Zembedding},
\[
 \limsup_{k\to\infty}
 \|\widetilde{\boldsymbol{\psi}}_{\mathcal O_k}\|_{H^1}
 \lesssim1.
\]

For $R>0$ and sufficiently large $k$, define
\[
 \boldsymbol\phi_k^R(y)
 =N_k^{-2}\eta(y/R)
   \bigl(\Pi_{-t_k,-x_k}\widetilde{\textbf{f}}_k\bigr)
        \bigl(\Psi(y/N_k)\bigr).
\]
These sequences are uniform bounded in $R$ in $\dot{H}^1(\mathbb{R}^6)$,
thus we obtain $
 \boldsymbol\phi_k^R\rightharpoonup\boldsymbol\phi^R$ in $\dot H^1(\mathbb R^6;\mathbb C^2)$.
Moreover, we assume $\boldsymbol{\phi^{R}}\rightharpoonup \boldsymbol{\phi}\in \dot{H}^{\frac{1}{2}}(\mathbb{R}^d)$, and by uniqueness, $\boldsymbol{\phi}^R=\eta(\frac{x}{R})\boldsymbol{\phi}(x)$.

Let $\boldsymbol\gamma\in C_c^\infty(B(0,R/2);\mathbb C^2)$.
For $j=1,2$, we have 
\begin{equation}\label{weak-profile-test}
\begin{aligned}
 \left\langle(\widetilde{\textbf{f}}_k)_j,
       (\widetilde{\boldsymbol\gamma}_{\mathcal O_k})_j
       \right\rangle_{H^1(\mathbb T^6)}=
 \left\langle(\boldsymbol\phi_k^R)_j,\gamma_j
       \right\rangle_{\dot H^1(\mathbb R^6)}+o_k(1)
 \longrightarrow
 \langle\phi_j,\gamma_j\rangle_{\dot H^1(\mathbb R^6)}.
\end{aligned}
\end{equation}

Consequently,
$\textbf{h}_k=\widetilde{\textbf{f}}_k
               -\widetilde{\boldsymbol\phi}_{\mathcal O_k}$
is absent from $\mathcal O$.
Using \eqref{weak-profile-test} with
$\boldsymbol\gamma=\boldsymbol\psi$ in
\eqref{fixed-profile-test}, we obtain
\begin{equation}\label{profile-lower}
 \|\boldsymbol\phi\|_{\dot H^1}
 \gtrsim\Omega(\{\widetilde{\textbf{f}}_k\})>\delta.
\end{equation}
Expanding the squared norms, for $j=1,2$, gives
\begin{equation}\label{one-extraction}
\begin{aligned}
 \|(\textbf{h}_k)_j\|_2^2
 &=\|(\widetilde{\textbf{f}}_k)_j\|_2^2+o_k(1),\\
 \|\nabla(\textbf{h}_k)_j\|_2^2
 &=\|\nabla(\widetilde{\textbf{f}}_k)_j\|_2^2
    -\|\nabla\phi_j\|_{L^2(\mathbb R^6)}^2+o_k(1).
\end{aligned}
\end{equation}

Finally, set $\textbf{f}_k^0=\widetilde{\textbf{f}}_k$.
Whenever $\Omega(\{\textbf{f}_k^\alpha\})>\delta$, repeat
the preceding construction, retaining the selection
\eqref{selected-point}, and let
\[
 \textbf{f}_k^{\alpha+1}=\textbf{f}_k^\alpha
  -\widetilde{\boldsymbol\psi}^{\alpha+1}_{\mathcal O_k^{\alpha+1}}.
\]
So $\textbf{f}_k^{\alpha+1}$ is absent from $\mathcal{O}^{\alpha}$, and it also absent from previous frames. Also, all frames are pairwise orthogonal.
Iterating \eqref{one-extraction}, we have
\[
\begin{aligned}
 \|(\textbf{f}_k)_j\|_2^2
 &=
 \|g_j\|_2^2
 +\|(\textbf{f}_k^{\alpha+1})_j\|_2^2
 +o_k(1),\\
 \|\nabla(\textbf{f}_k)_j\|_2^2
 &=
 \|\nabla g_j\|_2^2
 +\sum_{\beta=1}^{\alpha+1}
   \|\nabla\psi_j^\beta\|_{L^2(\mathbb R^6)}^2
 +\|\nabla(\textbf{f}_k^{\alpha+1})_j\|_2^2
 +o_k(1).
\end{aligned}
\]
By \eqref{bound}, \eqref{profile-lower}, and the orthogonality property, the procedure stops after at most
$O(\delta^{-2})$ steps.
Suppose that $J$ profiles have been extracted when it stops.
We then define
$\textbf{R}_k
 =\textbf{f}_k-\textbf{g}
 -\sum_{\alpha=1}^{J}
   \widetilde{\boldsymbol{\psi}}^\alpha_{\mathcal O_k^\alpha}$.
By construction, $\Omega(\{\textbf{R}_k\})\le\delta$.
\Cref{refined} gives
\[
 \limsup_{k\to\infty}
 \|e^{it\Delta}\textbf{R}_k\|_{Z(I_k)}
 \lesssim \delta^{O(1)}.
\]
This completes the proof.
\end{proof}

\begin{remark}
The functions $\textbf{g}$ and
$\widetilde{\boldsymbol\psi}^{\alpha}_{\mathcal O_k^\alpha}$
are called profiles. The former is the scale-one profile,
while the latter are Euclidean profiles.
\end{remark}

\begin{proof}[Proof of Proposition~\ref{decomp}]
Let $\delta=2^{-l}$, and apply inductively \Cref{extraction} to get the sequences of frames and profiles $\mathcal{O}^{\alpha}$, $\tilde{\boldsymbol{\psi}}^{\alpha}_{\mathcal{O}^{\alpha}_k}$. And \eqref{energy-decoupling} following from orthogonality of frames.
\end{proof}

\begin{proposition}[Almost orthogonality of nonlinear profiles]
\label{profile-orth}
Let $\mathcal O^\alpha$ and $\mathcal O^\beta$ be orthogonal
renormalized Euclidean frames, with associated data
$\boldsymbol{\psi}^\alpha,\boldsymbol{\psi}^\beta
\in\dot H^1(\mathbb R^6;\mathbb C^2)$.
Assume that their associated Euclidean solutions satisfy
the scattering hypothesis of \Cref{prop:euclidean-profiles}.
Let $\textbf{U}_k^\alpha,\textbf{U}_k^\beta$ be the corresponding
nonlinear profiles supplied by that proposition, with
\[
 \textbf{U}_k^\alpha(0)
 =\widetilde{\boldsymbol{\psi}}^\alpha_{\mathcal O_k^\alpha},
 \qquad
 \textbf{U}_k^\beta(0)
 =\widetilde{\boldsymbol{\psi}}^\beta_{\mathcal O_k^\beta}.
\]
Let $\textbf{G}$ be the strong solution with
$\textbf{G}(0)=\textbf{g}\in H^1(\mathbb T^6;\mathbb C^2)$,
and denote its lifespan by $I_0$.
Suppose $I_k$ are intervals containing zero with $|I_k|\to0$.
Then, for $j=1,2$,
\begin{align}
 \lim_{k\to\infty}\sup_{t\in I_k}
 \left|
 \left\langle
   \nabla(\textbf{U}_k^\alpha)_j(t),
   \nabla(\textbf{U}_k^\beta)_j(t)
 \right\rangle_{L^2}
 \right|
 &=0,
 \label{nonlinear-gradient-orth}\\
 \lim_{k\to\infty}\sup_{t\in I_k\cap I_0}
 \left|
 \left\langle
   \nabla(\textbf{U}_k^\alpha)_j(t),
   \nabla G_j(t)
 \right\rangle_{L^2}
 \right|
 &=0.
 \label{nonlinear-scale-one-orth}
\end{align}
\end{proposition}
\begin{proof}
Set $\textbf{U}_k^0(0)=\textbf{g}$ and
$\textbf{U}_k^0=\textbf{G}$ for all $k$, and associate this
scale-one profile with the trivial frame $(1,0,0)_k$.

Fix $j\in\{1,2\}$, for any $\theta>0$, \Cref{euclidean} gives
$T_{\theta,\alpha},R_{\theta,\alpha},T_{\theta,\beta},
R_{\theta,\beta}$ sufficiently large such that
\[
\begin{aligned}
 \textbf{U}_k^\alpha
 &=\boldsymbol{\omega}_k^{\alpha,\theta,-\infty}
  +\boldsymbol{\omega}_k^{\alpha,\theta,+\infty}
  +\boldsymbol{\omega}_k^{\alpha,\theta}
  +\boldsymbol{\rho}_k^{\alpha,\theta},\\
 \textbf{U}_k^\beta
 &=\boldsymbol{\omega}_k^{\beta,\theta,-\infty}
  +\boldsymbol{\omega}_k^{\beta,\theta,+\infty}
  +\boldsymbol{\omega}_k^{\beta,\theta}
  +\boldsymbol{\rho}_k^{\beta,\theta}.
\end{aligned}
\]
For $\textbf{U}_k^0$, set
 $\boldsymbol{\rho}_k^{0,\theta}
 =\boldsymbol{\omega}_k^{0,\theta}=0, 
 \boldsymbol{\omega}_k^{0,\theta,+\infty}
 =\boldsymbol{\omega}_k^{0,\theta,-\infty}
 =\tfrac12\textbf{G}.$ Thus the pairing with $\textbf{G}$ is included by taking
$\beta=0$. Since \eqref{eq:euclidean-profile-core}
and \eqref{eq:euclidean-profile-tail} still give
$\|\boldsymbol{\rho}_k^{\gamma,\theta}\|_{L_t^\infty H_x^1}
\le C\theta$, it suffices to consider the following three
pairings, uniformly for $t\in I_k$:
\[
 \left\langle
 (\boldsymbol{\omega}_k^{\alpha,\theta,\pm\infty})_j,
 (\boldsymbol{\omega}_k^{\beta,\theta,\pm\infty})_j
 \right\rangle_{\dot H^1},
 \left\langle
 (\boldsymbol{\omega}_k^{\alpha,\theta,\pm\infty})_j,
 (\boldsymbol{\omega}_k^{\beta,\theta})_j
 \right\rangle_{\dot H^1},
 \left\langle
 (\boldsymbol{\omega}_k^{\alpha,\theta})_j,
 (\boldsymbol{\omega}_k^{\beta,\theta})_j
 \right\rangle_{\dot H^1}.
\]

\medskip\noindent
\emph{Case 1.}
By the construction, for $\gamma=\alpha,\beta\ne0$,
\[
 \boldsymbol{\omega}_k^{\gamma,\theta,\pm\infty}(t)
 =\mathbf{1}_{\{\pm(t-t_k^\gamma)\ge
       T_{\theta,\gamma}(N_k^\gamma)^{-2},\,
       |t|\le T_{\theta,\gamma}^{-1}\}}
   \Pi_{t_k^\gamma-t,x_k^\gamma}T_{N_k^\gamma}
   \boldsymbol{\phi}^{\gamma,\pm\infty},
\]
where $\boldsymbol{\phi}^{\gamma,\pm\infty}$ are the Euclidean
scattering states. So by \eqref{linear-orth}
\[
\begin{aligned}
 \sup_{t\in I_k}\left|\left\langle
 (\boldsymbol{\omega}_k^{\alpha,\theta,\pm\infty})_j(t),
 (\boldsymbol{\omega}_k^{\beta,\theta,\pm\infty})_j(t)
 \right\rangle_{\dot H^1}\right| \le\left|\left\langle
 (\Pi_{t_k^\alpha,x_k^\alpha}T_{N_k^\alpha}
       \boldsymbol{\phi}^{\alpha,\pm\infty})_j,
 (\Pi_{t_k^\beta,x_k^\beta}T_{N_k^\beta}
       \boldsymbol{\phi}^{\beta,\pm\infty})_j
 \right\rangle_{\dot H^1}\right|\longrightarrow0.
\end{aligned}
\]
\noindent\emph{Case 2.}
By the construction in \Cref{euclidean},
\[
 \boldsymbol{\omega}_k^{\beta,\theta}
 =\widetilde{\textbf{u}}_k^\beta\mathbf{1}_{S_k^{\beta,\theta}},
 \qquad
 S_k^{\beta,\theta}
 =\left\{(t,x):
 \begin{array}{l}
 |t-t_k^\beta|<T_{\theta,\beta}(N_k^\beta)^{-2},\\
 d(x,x_k^\beta)\le R_{\theta,\beta}(N_k^\beta)^{-1}
 \end{array}\right\}.
\]
By a similar argument in \cite{YY}, we get 
 $ \left\langle (\boldsymbol{\omega}_k^{\alpha,\theta,\pm\infty})_j,
 (\boldsymbol{\omega}_k^{\beta,\theta})_j
 \right\rangle_{\dot H^1}\to 0 $
 as $\theta \to 0$.

\noindent\emph{Case 3.}
Let $\varepsilon>0$ be small. If $
 \frac{N_k^\alpha}{N_k^\beta}
 +\frac{N_k^\beta}{N_k^\alpha}\le\varepsilon^{-1000}$,
then $S_k^{\alpha,\theta}\cap S_k^{\beta,\theta}=\varnothing$
for all sufficiently large $k$ by orthogonality of the frames.
Thus the pairing is zero.
Otherwise, by symmetry, we may assume that
$N_k^\alpha/N_k^\beta\ge\varepsilon^{-1000}/2$.
By Bernstein estimate and derivative bound of $\omega_k^{\beta,\theta}$,
\[
\begin{aligned}
 \|P_{\le\varepsilon^{10}N_k^\alpha}
          \boldsymbol{\omega}_k^{\alpha,\theta}\|_{\dot H^1}
 &\lesssim(\varepsilon^{10}N_k^\alpha)^4
          \|\boldsymbol{\omega}_k^{\alpha,\theta}\|_{L^1}
 \lesssim \varepsilon^{40},\\
 \|P_{>\varepsilon^{-10}N_k^\beta}
          \boldsymbol{\omega}_k^{\beta,\theta}\|_{\dot H^1}
 &\lesssim(\varepsilon^{-10}N_k^\beta)^{-1}
          \|\nabla^2\boldsymbol{\omega}_k^{\beta,\theta}\|_{L^2}
 \lesssim \varepsilon^{10}.
\end{aligned}
\]
Since $\varepsilon^{10}N_k^\alpha\gg
\varepsilon^{-10}N_k^\beta$, the remaining frequency supports
are disjoint. Consequently,
\[
\begin{aligned}
 \left|\left\langle
 (\boldsymbol{\omega}_k^{\alpha,\theta})_j,
 (\boldsymbol{\omega}_k^{\beta,\theta})_j
 \right\rangle_{\dot H^1}\right|& \le
 \left|\left\langle
 (P_{\le\varepsilon^{10}N_k^\alpha}
       \boldsymbol{\omega}_k^{\alpha,\theta})_j,
 (\boldsymbol{\omega}_k^{\beta,\theta})_j
 \right\rangle_{\dot H^1}\right|\\
 &\quad +
 \left|\left\langle
 (P_{>\varepsilon^{10}N_k^\alpha}
       \boldsymbol{\omega}_k^{\alpha,\theta})_j,
 (P_{>\varepsilon^{-10}N_k^\beta}
       \boldsymbol{\omega}_k^{\beta,\theta})_j
 \right\rangle_{\dot H^1}\right|\\
 &\quad +
 \left|\left\langle
 (P_{>\varepsilon^{10}N_k^\alpha}
       \boldsymbol{\omega}_k^{\alpha,\theta})_j,
 (P_{\le\varepsilon^{-10}N_k^\beta}
       \boldsymbol{\omega}_k^{\beta,\theta})_j
 \right\rangle_{\dot H^1}\right|\\
 &\lesssim \varepsilon^{40}+\varepsilon^{10}.
\end{aligned}
\]
Sending $k\to\infty$ and then
$\varepsilon\to0$ proves this case.

Combining the three cases complete the proof.
\end{proof}
\section{Proof of \Cref{thm:main}}\label{se6}
First, we consider the focusing threshold\cite{GMXZ},
\begin{proposition}\label{sharp}
Recall the definition of threshold in \eqref{ground}, we obtain
\begin{equation}\label{ground-identities}
-\Delta \mathbf{W}=F(\mathbf{W}),\qquad
3E(\mathbf{W})=\|\nabla W_1\|_2^2+\frac14\|\nabla W_2\|_2^2
=\frac34\int_{\mathbb R^6}Q^3.
\end{equation}
\end{proposition}
By a simple energy trapping, we have
\begin{proposition}\label{trap}
Mass and energy are conserved by strong solutions. If
\[
E(\mathbf{u}_0)+c_0 M(\mathbf{u}_0)\le(1-\delta_0)E(\mathbf{W}),\quad
\|\nabla u_{0,1}\|_2^2+\frac14\|\nabla u_{0,2}\|_2^2+c_0 M(\mathbf{u}_0)<3E(\mathbf{W})
\]
for $0<\delta_0<1$ and $c_0$ in \eqref{eq:torus-sharp-sobolev}, then throughout their lifespan,
\begin{align}
\|\nabla u_1(t)\|_2^2+\frac14\|\nabla u_2(t)\|_2^2+c_0 M(\mathbf{u}(t))
&\le(1-\sqrt{\delta_0/3})^2\,3E(\mathbf{W}),\label{trapping}\\
(\|\nabla u_1(t)\|_2^2+\frac14\|\nabla u_2(t)\|_2^2+c_0 M(\mathbf{u}(t)))&\le 3 (E(\mathbf{u}(t))+c_0M(\mathbf{u}(t))) .\label{coercivity}
\end{align}
\end{proposition}
For periodic data $\mathbf{v}$, we write
\[
L(\mathbf{v})=E(\mathbf{v})+c_0M(\mathbf{v}).
\]

\begin{definition}\label{Lambda}
For $0<L<E(\mathbf{W})$ and $\tau>0$, define
\begin{equation*}
\begin{split}
\Lambda(L,\tau)=\sup\big\{&\|\mathbf{u}\|_{Z(I)}:\ \mathbf{u}\text{ is a strong solution of \eqref{eq:system}},\ |I|\le\tau,\\
&E(\mathbf{u}(t))+c_0M(\mathbf{u}(t))<L,\\
&\|\nabla u_1(t)\|_2^2+\tfrac14\|\nabla u_2(t)\|_2^2+c_0M(\mathbf{u}(t))<3E(\mathbf{W}),\ t\in I\big\}.
\end{split}
\end{equation*}
Set
\begin{equation}\label{Emax}
\Lambda_0(L)=\lim_{\tau\downarrow0}\Lambda(L,\tau),\qquad
E_{\max}=\sup\{L\in(0,E(\mathbf{W})):\Lambda_0(L)<\infty\}.
\end{equation}
\end{definition}
By the definition,
\[
\Lambda(L,\tau+\sigma)\le\Lambda(L,\tau)+\Lambda(L,\sigma).
\]
Consequently,
\[
\Lambda_0(L)<\infty
\quad\Longleftrightarrow\quad
\Lambda(L,\tau)<\infty\quad\hbox{for every }\tau>0.
\]

\begin{theorem}\label{global-critical}
The number in \eqref{Emax} satisfies $E_{\max}=E(\mathbf{W})$.
\end{theorem}
\begin{proof}
\medskip\noindent\textit{Step 1.}
Suppose, to the contrary, that
\begin{equation}\label{critical-assume}
0<E_{\max}<E(\mathbf{W}).
\end{equation}
Choose $E_{\max}<L_k<E(\mathbf{W})$ with $L_k\downarrow E_{\max}$.
By definition, there exists a sequence of $\mathbf{u}_k$ on intervals $I_k$ such that
\begin{equation}\label{critical-sequence}
\begin{gathered}
I_k=[-T_k,T^k],\qquad T_k+T^k\le k^{-1},\\
L(\mathbf{u}_k(0))\longrightarrow E_{\max},\quad
\|\mathbf{u}_k\|_{Z([-T_k,0])}\longrightarrow\infty,\quad
\|\mathbf{u}_k\|_{Z([0,T^k])}\longrightarrow\infty.
\end{gathered}
\end{equation}

Choose $\delta_0\in(0,1)$ such that
$E_{\max}<(1-\delta_0)E(\mathbf{W})$. For all sufficiently large $k$,
$L(\mathbf{u}_k(0))\le(1-\delta_0)E(\mathbf{W})$. \Cref{trap}, with
$1-\delta=(1-\sqrt{\delta_0/3})^2<1$, yields
\begin{equation}\label{critical-kinetic}
\sup_{k,t\in I_k}
\bigl(\|\nabla u_{k,1}(t)\|_2^2
 +\tfrac14\|\nabla u_{k,2}(t)\|_2^2+c_0M(\mathbf{u}_k(t))\bigr)
\le(1-\delta)3E(\mathbf{W}).
\end{equation}
In particular, $\mathbf{u}_k(0)$
is bounded in $H^1$ space.

\medskip\noindent\textit{Step 2.}
Apply \Cref{decomp} to $\mathbf{u}_k(0)$ and $I_k$.
After passing to a subsequence, for each fixed finite $J$,
\begin{equation}\label{critical-decomp}
\mathbf{u}_k(0)=\mathbf{g}+\sum_{\alpha=1}^J\widetilde{\boldsymbol{\psi}}_k^\alpha+\mathbf{R}_k^J,
\quad
\widetilde{\boldsymbol{\psi}}_k^\alpha
 =\Pi_{t_k^\alpha,x_k^\alpha}T_{N_k^\alpha}\boldsymbol{\psi}^\alpha,\quad
\lim_{J\to J^*}\limsup_{k\to\infty}
\|e^{it\Delta}\mathbf{R}_k^J\|_{Z(I_k)}=0.
\end{equation}
By
\eqref{Zembedding}, 
\begin{equation}\label{mass-to-scale-one}
M(\widetilde{\boldsymbol{\psi}}_k^\alpha)
\lesssim(N_k^\alpha)^{-1}\|\boldsymbol{\psi}^\alpha\|_{\dot H^1}^2\longrightarrow0,
\end{equation}
thus through a simple computation, we get
\begin{align}
\|\nabla u_{k,1}(0)\|_2^2+&\frac14\|\nabla u_{k,2}(0)\|_2^2+c_0M(\mathbf{u}_k(0))
=\|\nabla g_1\|_2^2+\frac14\|\nabla g_2\|_2^2+c_0M(\mathbf{g})\notag \\
& +\sum_{\alpha=1}^J\bigl(\|\nabla\psi_1^\alpha\|_2^2
                  +\frac14\|\nabla\psi_2^\alpha\|_2^2\bigr)
 +\|\nabla R_{k,1}^J\|_2^2
 +\frac14\|\nabla R_{k,2}^J\|_2^2+c_0M(\mathbf{R}_k^J)+o_k(1).
\label{corrected-kinetic-splitting}
\end{align}
Therefore, by \eqref{critical-kinetic}, for large $k$,
\[
\begin{gathered}
H(\mathbf g)+c_0M(\mathbf g)<3E(\mathbf W),\\
H(\boldsymbol\psi^\alpha)<3E(\mathbf W),\\
H(\mathbf R_k^J)+c_0M(\mathbf R_k^J)<3E(\mathbf W)
\end{gathered}
\]
Thus the sharp inequality \eqref{eq:torus-sharp-sobolev} gives
\begin{equation}\label{profile-corrected-coercivity}
L(\mathbf{v})\ge\frac13\bigl(\|\nabla v_1\|_2^2
 +\frac14\|\nabla v_2\|_2^2+c_0M(\mathbf{v})\bigr)\ge0
\end{equation}
for $\mathbf{v}=\mathbf{g}$ and, for fixed $J$ and large $k$, $\mathbf{v}=\mathbf{R}_k^J$.
Consequently, \eqref{energy-decoupling} yields
\begin{equation}\label{energy-sum}
L(\mathbf{g})+\sum_{\alpha=1}^JL(\alpha)+\lim_{k\to\infty}L(\mathbf{R}_k^J)=E_{\max},
\end{equation}
where all contributions are nonnegative and
\begin{equation}\label{Lalpha}
L(\alpha):=\lim_{k\to \infty} L(\widetilde{\boldsymbol{\psi}}_k^\alpha)
=\begin{cases}
E_{\mathbb R^6}(\boldsymbol{\psi}^\alpha),&t_k^\alpha=0,\\
\|\nabla\psi_1^\alpha\|_2^2+\tfrac14\|\nabla\psi_2^\alpha\|_2^2,
&|(N_k^\alpha)^2t_k^\alpha|\to\infty.
\end{cases}
\end{equation}
By \eqref{corrected-kinetic-splitting},\eqref{profile-corrected-coercivity},\eqref{energy-sum} and \eqref{Lalpha}, we notice that for $t^{\alpha}_k=0$
\begin{equation}\label{euclidean-inherit}
E_{\mathbb R^6}(\boldsymbol{\psi}^\alpha)\le E_{\max}<E(\mathbf{W}),\qquad
\|\nabla\psi_1^\alpha\|_2^2+\tfrac14\|\nabla\psi_2^\alpha\|_2^2<3E(\mathbf{W}),
\end{equation}
and for $|(N^{\alpha}_k)^2t^{\alpha}_k|\to \infty$
\begin{equation}\label{wave-inherit}
\|\nabla\psi_1^\alpha\|_2^2+\tfrac14\|\nabla\psi_2^\alpha\|_2^2
=L(\alpha)\le E_{\max}<E(\mathbf{W}).
\end{equation}
Hence apply \Cref{prop:euclidean-profiles} to every Euclidean
profile. Let $\mathbf{G}(0)=\mathbf{g}$ and $\mathbf{U}_k^\alpha(0)=\widetilde{\boldsymbol{\psi}}_k^\alpha$
be the corresponding periodic solutions.

\medskip\noindent\textit{Step 3.} We consider following four cases.

\medskip\noindent\textit{Case 1: No profiles.}
Here $\mathbf{g}=0$, $J^*=0$, and $\mathbf{u}_k(0)=\mathbf{R}_k^0$. Since
\[
\|e^{it\Delta}\mathbf{R}_k^0\|_{Z'(I_k)}
=\bigl(\|e^{it\Delta}\mathbf{R}_k^0\|_{Z(I_k)}
       \|e^{it\Delta}\mathbf{R}_k^0\|_{X^1(I_k)}\bigr)^{1/2}\longrightarrow0,
\]
So \[
\|\mathbf{u}_k\|_{Z(I_k)}\lesssim \|\mathbf{u}_k\|_{X^1(I_k)}\lesssim 1
\]by the local theory, this contradicts \eqref{critical-sequence}.

\medskip\noindent\textit{Case 2: Exactly one Euclidean profile.}
Suppose that $\mathbf{g}=0$ and $J^*=1$. By
\Cref{decomp}, we have
\[
\mathbf{u}_k(0)=\widetilde{\boldsymbol{\psi}}_k^1+\mathbf{R}_k^1,
\qquad
\sup_k\|\mathbf{R}_k^1\|_{H^1}<\infty,
\qquad
\|e^{it\Delta}\mathbf{R}_k^1\|_{Z(I_k)}\longrightarrow0.
\]
The linear estimate in \Cref{prop-pre-Xs-Hs}
and \eqref{eq:z-prime-norm} imply
\[
\|e^{it\Delta}\mathbf{R}_k^1\|_{X^1(I_k)}
\lesssim\|\mathbf{R}_k^1\|_{H^1}\le C,
\qquad
\|e^{it\Delta}\mathbf{R}_k^1\|_{Z'(I_k)}\longrightarrow0.
\]

Let $\mathbf{U}_k^1$ be the periodic nonlinear solution with initial
data $\mathbf{U}_k^1(0)=\widetilde{\boldsymbol{\psi}}_k^1$. The threshold conditions
\eqref{euclidean-inherit} and \eqref{wave-inherit}
allow us to apply \Cref{prop:euclidean-profiles}.
Thus there exist $\tau>0$ and $C<\infty$, independent of $k$,
such that, for all sufficiently large $k$, $
\|\mathbf{U}_k^1\|_{X^1((-\tau,\tau))}\le C$.
Since $0\in I_k$ and $|I_k|\to0$, we have
$I_k\subset(-\tau,\tau)$ for large $k$, and consequently
$\|\mathbf{U}_k^1\|_{X^1(I_k)}\le C$.

Set $
\mathbf{U}_{\mathrm{app},k}:=\mathbf{U}_k^1+e^{it\Delta}\mathbf{R}_k^1$, 
then $
\mathbf{U}_{\mathrm{app},k}(0)=\mathbf{u}_k(0)$ and 
$\|\mathbf{U}_{\mathrm{app},k}\|_{X^1(I_k)}\le C.$
Let $\mathbf{U}_{\mathrm{app},k}$ be the solution of \eqref{error}, so the error is
\begin{align*}
\mathbf{e}_k
:=(i\partial_t+\Delta)\mathbf{U}_{\mathrm{app},k}
       +F(\mathbf{U}_{\mathrm{app},k}) =F(\mathbf{U}_k^1+e^{it\Delta}\mathbf{R}_k^1)-F(\mathbf{U}_k^1).
\end{align*}
By \Cref{lem:nonlinear-estimates},
\begin{align*}
\|\mathbf{e}_k\|_{N(I_k)}
\lesssim
\bigl(\|\mathbf{U}_k^1\|_{Z'(I_k)}
      +\|e^{it\Delta}\mathbf{R}_k^1\|_{Z'(I_k)}\bigr)
       \|e^{it\Delta}\mathbf{R}_k^1\|_{Z'(I_k)} \longrightarrow0,
\end{align*}
where
$\|\mathbf{U}_k^1\|_{Z'(I_k)}
 \lesssim\|\mathbf{U}_k^1\|_{X^1(I_k)}\le C$
by \Cref{cor-strichartz}, \eqref{pre:Zformula} and \eqref{eq:z-prime-norm}.

Moreover, 
\[
\|\mathbf{U}_{\mathrm{app},k}\|_{Z(I_k)}
+\|\mathbf{U}_{\mathrm{app},k}\|_{L_t^\infty H_x^1(I_k)}
\le M.
\]
So by 
\Cref{prop:local-theory-h1},
\[
\|\mathbf{u}_k-\mathbf{U}_{\mathrm{app},k}\|_{X^1(I_k)}
\longrightarrow0.
\]
It follows that
\[
\|\mathbf{u}_k\|_{Z(I_k)}
\lesssim\|\mathbf{u}_k\|_{X^1(I_k)}
\le\|\mathbf{U}_{\mathrm{app},k}\|_{X^1(I_k)}+o_k(1)
\le C+o_k(1),
\]
contradicting \eqref{critical-sequence}.

\medskip\noindent\textit{Case 3: Exactly one scale-one profile.}
Here $\mathbf{g}\ne0$ and $J^*=0$. The local solution $\mathbf{G}$ exists on a
fixed neighborhood containing $I_k$ for large $k$. Applying the similar
argument of Case~2 to $\mathbf{G}+e^{it\Delta}\mathbf{R}_k^0$ gives the same
contradiction.

\medskip\noindent\textit{Case 4: Multiple profiles.}
Write $\mathbf{U}_k^0=\mathbf{G}$ and
$\mathbf{U}_{\mathrm{prof},k}^J=\sum_{\alpha=0}^J\mathbf{U}_k^\alpha$.

\begin{lemma}\label{claim-sum}
There exists $Q<\infty$, independent of $J$, such that
\begin{equation}\label{claim-bound}
\sup_{J<\infty} \limsup_{k\to\infty}
\left(
\|\mathbf{U}_{\mathrm{prof},k}^J\|_{X^1(I_k)}^2
+\sum_{\alpha=0}^J\|\mathbf{U}_k^\alpha\|_{X^1(I_k)}^2
\right)\le Q^2.
\end{equation}

\end{lemma}

\begin{proof}
By \eqref{gradient-pythagoras},
$\sum_{\alpha=1}^{J^*}\|\boldsymbol{\psi}^\alpha\|_{\dot H^1}^2<\infty$.
Choose $A$ so that every profile with $\alpha>A$ satisfied $\|\boldsymbol{\psi}^\alpha\|_{\dot H^1}^2 \le \varepsilon$ for small $\varepsilon$.
So for $\alpha>A$, \Cref{prop:local-theory-h1} gives
\begin{align}
\limsup_k\|\mathbf{U}_k^\alpha\|_{X^1(I_k)}
&\lesssim\|\boldsymbol{\psi}^\alpha\|_{\dot H^1},
\label{smallprofile}\\
\limsup_k
\|\mathbf{U}_k^\alpha-e^{it\Delta}\widetilde{\boldsymbol{\psi}}_k^\alpha\|_{X^1(I_k)}
&\lesssim\|\boldsymbol{\psi}^\alpha\|_{\dot H^1}^2.
\label{smallcorrection}
\end{align}
And for $\alpha < J$, by \Cref{prop:euclidean-profiles}, $\sum_{\alpha=0}^{A}
\limsup_k\|\mathbf U_k^\alpha\|_{X^1(I_k)}^2<\infty.$
By triangle inequality,
\begin{align*}
&\quad \limsup_k\|\mathbf{U}_{\mathrm{prof},k}^J\|_{X^1(I_k)}\\
&\le\sum_{\alpha=0}^A\limsup_k\|\mathbf{U}_k^\alpha\|_{X^1(I_k)} +\sum_{\alpha=A+1}^J\limsup_k
\|\mathbf{U}_k^\alpha-e^{it\Delta}\widetilde{\boldsymbol{\psi}}_k^\alpha\|_{X^1(I_k)}
+\limsup_k\left\|e^{it\Delta}
\sum_{\alpha=A+1}^J\widetilde{\boldsymbol{\psi}}_k^\alpha\right\|_{X^1(I_k)}\\
&\lesssim\sum_{\alpha=0}^A\limsup_k\|\mathbf{U}_k^\alpha\|_{X^1(I_k)}
+\sum_{A<\alpha\le J^*}\|\boldsymbol{\psi}^\alpha\|_{\dot H^1}^2 +\left(\sum_{A<\alpha\le J^*}
\|\boldsymbol{\psi}^\alpha\|_{\dot H^1}^2\right)^{1/2}\\
&< \infty.
\end{align*}
Similarly,
\[
\limsup_k\sum_{\alpha=0}^J\|\mathbf{U}_k^\alpha\|_{X^1(I_k)}^2
\lesssim
\sum_{\alpha=0}^A\limsup_k\|\mathbf{U}_k^\alpha\|_{X^1(I_k)}^2
+\sum_{A<\alpha\le J^*}\|\boldsymbol{\psi}^\alpha\|_{\dot H^1}^2< \infty.
\]
Choosing $Q$ sufficiently large proves the lemma.
\end{proof}
Define
\begin{equation}\label{nonlinear-app}
\mathbf{U}_{\mathrm{app},k}^J=\mathbf{U}_{\mathrm{prof},k}^J+e^{it\Delta}\mathbf{R}_k^J,
\qquad
\mathbf{e}_k^J=(i\partial_t+\Delta)\mathbf{U}_{\mathrm{app},k}^J+F(\mathbf{U}_{\mathrm{app},k}^J).
\end{equation}
Then
\begin{equation}\label{exact-initial}
\mathbf{U}_{\mathrm{app},k}^J(0)=\mathbf{u}_k(0),\qquad
\mathbf{e}_k^J=F(\mathbf{U}_{\mathrm{prof},k}^J+e^{it\Delta}\mathbf{R}_k^J)
 -\sum_{\alpha=0}^JF(\mathbf{U}_k^\alpha).
\end{equation}
By \Cref{claim-sum} and embedding, we have 
\begin{equation}\label{uniform-approximation-bound}
\sup_J\limsup_k\left(
\|\mathbf{U}_{\mathrm{app},k}^J\|_{Z(I_k)}
+\|\mathbf{U}_{\mathrm{app},k}^J\|_{L_t^\infty H_x^1(I_k)}\right)\lesssim  Q.
\end{equation}

\begin{lemma}\label{residual}
The approximate solutions satisfy
\begin{equation}\label{residual-limit}
\lim_{J\to J^*}\limsup_k\|\mathbf{e}_k^J\|_{N(I_k)}=0.
\end{equation}
\end{lemma}
\begin{proof}
      This is similar to proof of Lemma 6.4 in \cite{YY} 
\end{proof}

We now apply \Cref{prop:stability-h1} with $\varepsilon(Q)$.
By \eqref{residual-limit}, we can fix a finite $J$ such that
\[
\limsup_{k\to\infty}\|\mathbf{e}_k^J\|_{N(I_k)}
\lesssim \varepsilon(Q).
\]
If $J^*<\infty$, take $J=J^*$. With this $J$ fixed,
\eqref{uniform-approximation-bound} and \eqref{residual} imply that, for all sufficiently large $k$,
\[
\|\mathbf{U}_{\mathrm{app},k}^J\|_{Z(I_k)}
+\|\mathbf{U}_{\mathrm{app},k}^J\|_{L_t^\infty H_x^1(I_k)}
\lesssim Q,
\qquad
\|\mathbf{e}_k^J\|_{N(I_k)}\lesssim \varepsilon(Q).
\]
Moreover, \eqref{exact-initial} gives
\[
\|\mathbf{U}_{\mathrm{app},k}^J(0)-\mathbf{u}_k(0)\|_{H^1}=0.
\]
By \Cref{prop:stability-h1}, we have 
\[
\|\mathbf{u}_k-\mathbf{U}_{\mathrm{app},k}^J\|_{X^1(I_k)}
\lesssim \varepsilon Q,
\qquad
\|\mathbf{u}_k\|_{X^1(I_k)}\lesssim  Q.
\]
In particular,  \Cref{Yembedding} yield $
\|\mathbf{u}_k\|_{Z(I_k)}\lesssim Q,$
this contradicts
$\|\mathbf{u}_k\|_{Z(I_k)}\to\infty$ in
\eqref{critical-sequence}. Hence Case~4 is impossible,
and therefore $E_{\max}=E(\mathbf{W})$.
\end{proof}

\begin{proof}[Proof of \Cref{thm:main}]
It is easy to see that \Cref{global-critical} implies \Cref{thm:main}.
\end{proof}

\vspace{5mm}
\end{document}